\documentclass[final,leqno]{siamltex}

\usepackage{amsmath}
\usepackage{amssymb}
\usepackage[figuresright]{rotating}
\usepackage{tikz}
\usetikzlibrary{calc}
\usepackage{bm}
\usepackage{epstopdf}
\usepackage{graphicx}
\usepackage{booktabs}
\usepackage{subfigure}
\usepackage{multirow}
\usepackage{dsfont}
\usepackage{epstopdf}
\usepackage{stmaryrd} 
\usepackage{changepage}

\newtheorem{thm}{Theorem}[section]
\newtheorem{lema}{Lemma}[section]
\newtheorem{Assumption}{Assumption}[section]
\newtheorem{coro}{Corollary}[section]
\newtheorem{remark}{Remark}[section]
\newtheorem{exam}{Example}[section]

\usepackage{enumerate}
\usepackage{caption}
\usepackage{graphicx}
\usepackage{float} 
\usepackage[right=3cm]{geometry}

\usepackage{mathrsfs}  
\usepackage{epstopdf}
\usepackage{epsfig}

\usepackage{booktabs} 

\DeclareMathOperator{\dive}{div}

\title{A priori and a posteriori error estimates of a really pressure-robust virtual element method for the incompressible Brinkman problem}

\author{Yu Xiong$^\dag$ \and Yanping Chen$^\ast$}

\begin{document}

\maketitle

\begin{abstract}
This paper presents both a priori and a posteriori error analyses for  a really pressure-robust virtual element method to approximate the incompressible Brinkman problem.  We construct a divergence-preserving reconstruction operator using the Raviart–Thomas element for the discretization on the right-hand side. The optimal priori error estimates are carried out, which imply the velocity error in the energy norm is independent of both the continuous pressure and the viscosity. Taking advantage of the virtual element method's ability to handle more general polygonal meshes, we implement effective mesh refinement strategies and develop a residual-type a posteriori error estimator. This estimator is proven to provide global upper and local lower bounds for the discretization error.  Finally, some numerical experiments  demonstrate the robustness, accuracy, reliability and efficiency of the method.
\end{abstract}



\let\thefootnote\relax\footnotetext{\small 
    $\ast$ Corresponding author. School of Science, Nanjing University of Posts and Telecommunications, Nanjing 210023, China. \ ({\tt Email:yanpingchen@njupt.edu.cn}) \par
     $\dag$ School of Mathematics and Computational Science, Hunan Key Laboratory for Computation and Simulation in Science and Engineering, Xiangtan University, Xiangtan, 411105, Hunan, China.\ ({\tt Email:xiongyu@smail.xtu.edu.cn}) \par 
    
  This work is supported by the State Key Program of National Natural Science Foundation of China (11931003), Natural Science Research Start-up Foundation of Recruiting Talents of Nanjing
  University of Posts and Telecommunications (NY223127) and Postgraduate Scientific Research Innovation Foundation of Xiangtan University (XDCX2024Y178).
  
}

\pagestyle{myheadings}
\thispagestyle{plain}

\section{Introduction}	\label{introduction}
This paper is concerned with the development of robust virtual element numerical methods for the Brinkman equations. The Brinkman equations model fluid flow in complex porous media with a permeability coefficient highly varying so that the flow is dominated by Darcy in some regions and by Stokes in others. In a simple form, the Brinkman model seeks unknown velocity $\mathbf{u}$ and  pressure $p$ satisfying	
 \begin{subequations}\label{Brinkman equations}
     \begin{align}
           -\nu\Delta \mathbf{u} + \nu \kappa^{-1} \mathbf{u} - \nabla p &=\mathbf{f} \qquad \text{in} \;\Omega,	\label{Brinkman eq 1}\\
         \nabla\cdot \mathbf{u} &= 0 \qquad \text{in} \;\Omega,  \label{Brinkman eq 2} \\
         \mathbf{u} &= \mathbf{0}\qquad \text{on}\; \partial\Omega.
     \end{align}
 \end{subequations}	
Here,  $\nu>0$ and $\mathbf{f}$ denote the fluid viscosity and source term, respectively. The permeability tensor $\kappa$ of the porous media is symmetric positive definite.

The incompressible Brinkman problem, as given by equations \eqref{Brinkman eq 1} and \eqref{Brinkman eq 2}, can efficiently capture Stokes and Darcy type flow behavior based on $\kappa$ without  necessitating the complex interface conditions typically required by the Stokes-Darcy interface model. This  is very convenient when modeling complicated  porous media scenarios, with significant implications for industrial and environmental issues such as industrial filters and open foams \cite{mu2014}. The permeability, which varies greatly, causes flow velocity to change substantially through the porous medium. In regions with high permeability $\kappa$, the model behaves like the Stokes equations, while in areas with low permeability $\kappa$, it simplifies to the Darcy equations. Therefore, numerical schemes for the Brinkman equations must be carefully designed to handle both Stokes and Darcy flows, which is essential for accurately simulating fluid flow in porous media. To address these challenges, various strategies have been proposed by researchers. In works such as \cite{burman2005}, \cite{burman2007} and \cite{xie2008},  Burman and Xie et al. have introduced jumps penalization techniques. These methods stabilize the Crouzeix-Raviart-P0 finite elements or P1-P0 finite elements by penalizing the normal component of the velocity field or the pressure field, respectively. It is worth mentioning that most of the numerical methods for solving  Brinkman equations are rely on  triangular (simplicial) and quadrilateral meshes. However, there exists another numerical methods that offers great flexibility in meshing with  polygonal elements, such as weak Galerkin  method (WG), mimetic finite differences method (MFD) and virtual element method (VEM). In particular, Mu et al. \cite{mu2014} developed a stable weak Galerkin finite element method that provides a robust and accurate numerical scheme for both Darcy and Stokes dominated flows. Furthermore, Zhai et al. \cite{zhai2016new} presented a new weak Galerkin finite element scheme for the  model, which offers a flexible and efficient approach to solve the Brinkman equations in complex porous media by utilizing generalized functions and their weak derivatives. 
	
 As a generalization of the finite element method (FEM) and mimetic finite differences method  on polygonal mesh, the virtual element method was initially proposed by Beir{\~a}o da Veiga et al. \cite{veiga2013,veiga2014} and has attracted the attention of many researchers in recent years. The construction of the virtual element space involves a combination of a polynomial subspace and an additional non-polynomial virtual subspace. To approximate the non-polynomial parts of the discrete bilinear forms, suitable projection operators are employed, which rely only on the degrees of freedom linked to the virtual element space. Consequently, we can effectively handle meshes with more general polygonal elements, without the need for directly calculating non-polynomial functions. For a thorough description of the VEM, see \cite{Antonietti2022,Veiga2023,Mascotto2023}. In recent years, the VEM has attracted more and more researchers' attention to solve fluid flow problems, such as the Stokes problem \cite{chen2019divergence, veiga2017a, Frerichs2020, JianMeng2023},  the Navier-Stokes problem \cite{Irisarri2019, Liyang2023}, the Darcy-Stokes problem  \cite{Zhaokun2020, Zhaokun2024} and so on \cite{Xiong2024b, Veiga2021, ypchen2023, Meng2024, Wangyang2024}.  As we  discussed in previous work \cite{xiong2024}, the VEM basis functions are closely related to the partial differential equation (PDE) within each element, facilitating the construction of divergence-free virtual elements. Although divergence-free VEM have been developed to solve the Brinkman problem \cite{Vacca2018brinkman}, it is not really pressure-robust due to the fact that the right-hand side term is not carefully discretized \cite{Frerichs2020}. In the standard divergence-free VEM scheme, this consistency error enters the prior velocity estimate as the inverse of the viscosity $1 / \nu$, so locking phenomenon occurs when $\nu\rightarrow 0$. A significant contribution in the VEM approximate to  the Brinkman equations was made by Wang et al. \cite{wang2021Brinkman}, who developed a really pressure-robust VEM scheme by constructing a divergence-preserving  $\text{CW}_0$ reconstruction operator. Moreover, the reconstruction technique has also been applied to the conforming VEM \cite{Wanggang2021,Wanggang2023} and nonconforming VEM \cite{Liu2020,Liu2022}  for (Navier-) Stokes  problems. Inspired by the aforementioned works, we will construct a really pressure-robust virtual element method (PR-VEM)  based on  divergence-preserving Raviart-Thomas reconstruction operator to solve the Brinkman problem. From a computational perspective, in practical applications, the lowest-order virtual element is typically preferred, and our method is also applicable for the case  $k=1$.

On the other hand, in numerical approximate to PDE problems, it is crucial to use adaptive mesh refinement techniques guided by a posteriori error indicators. For instance, they ensure that the errors remain beneath a specified threshold at a reasonable computational expense, particularly when dealing with solutions that exhibit singular behavior. Due to the large flexibility of the meshes to which the VEM is applied, mesh adaptivity becomes an appealing feature since mesh refinement strategies can be implemented very efficiently. The field of a posteriori error analysis for the VEM has seen significant development. Beir{\~a}o da Veiga et al. \cite{Veiga2015} firstly developed a residual-based a posteriori error estimator for the $C^1$-conforming VEM applied to the Poisson problem. After that they focused on a posteriori error estimation for mesh adaptivity within the hp-VEM framework \cite{Veiga2019}. Moreover, Chi et al. \cite{chi2019} introduced a recovery-based a posteriori error estimation framework that is applicable to VEMs of any order. Lastly, some residual-type  posteriori error estimators for the virtual element approximation of the (Navier-) Stokes problem are presented in \cite{Manzini2024,Wang2020post,Wang2021post}. As far as we know, there is only one paper on the a posteriori error estimation of the Brinkman equation using virtual elements \cite{Mauricio2020}. It is natural that we develop an adaptive PR-VEM algorithm guided by a residual-type a posteriori error estimator for the Brinkman equations \eqref{Brinkman equations}. 

The aim of this work is to develop a  really pressure-robust and divergence-free VEM for the two-dimensional incompressible Brinkman problem. In addition to deriving the optimal a priori error estimate, we also give a posterior error estimator and prove its reliability and efficiency. It is noteworthy that, in contrast to the standard divergence-free VEM scheme, the velocity error within the prior error estimation is independent not only of the continuous pressure $p$ but also of the viscosity $\nu$, there by achieving really pressure robustness. This article also provides a mesh refinement strategy and an adaptive algorithm. Numerical experiments not only validate our theoretical analyses, but also apply (adaptive) PR-VEM to porous media areas with high contrast permeability and complex geometric areas.

The structure of this paper is outlined as follows: Section 2 explains essential notations, derives the weak formulation of the Brinkman equations, and introduces the Helmholtz decomposition. In Section 3, a framework of a really pressure-robust VEM is constructed and the presented scheme is proven to be well-posed. The optimal a posteriori error estimate is derived in Section 4. Section 5 present a accurate and effective posteriori error estimator driving the adaptive mesh refinement. In Section 6, we implement some numerical experiments. Finally, we offer our conclusions in Section 7.

\section{Notations and Preliminaries}
 We consider a convex  polygonal domain $\Omega$ within the two-dimensional Euclidean space $\mathbb{R}^2$, characterized by a Lipschitz boundary $\partial\Omega$. The $L^2(\Omega)$ norm and inner product will be denoted by $\|\cdot\|$ and $(\cdot,\cdot)$, respectively, while  all other norms will be labeled with subscripts or specifically defined. 
 
 Then we consider the natural functions spaces for velocity and pressure, respectively, by
$$
\mathbf{V}:=[H_0^1(\Omega)]^2 \quad \text { and } \quad Q:=L_0^2(\Omega)=\left\{q \in L^2(\Omega):(q, 1)=0\right\},
$$
where the space $\mathbf{V}$ equipped with the energy norm 
\begin{equation}	\label{energy norm}
    \interleave \mathbf{v} \interleave : = \left(\|\nabla \mathbf{v}\|^2 + \|\kappa^{-1/2} \mathbf{v}\|^2\right)^{\frac{1}{2}} \qquad \forall \mathbf{v}\in\mathbf{V},
\end{equation}
where the permeability $\kappa$ is defined in \eqref{Brinkman equations}. In $\mathbf{V}$, we have the Friedrichs inequality: there exits a constant $C_F$ depending only on $\Omega$ such that for any $\mathbf{\Phi}\in\mathbf{V}$,
\begin{equation}	\label{Friedrichs}
    \|\mathbf{\Phi}\|_1 \le C_F \|\nabla\mathbf{\Phi}\|.
\end{equation}

Moreover, we  introduce the kernel space
\begin{equation}	\label{continuous kernel space}
       \mathbf{Z} := \{\mathbf{v}\in\mathbf{V}\quad \text{s.t.}\quad (\nabla\cdot\mathbf{v}, q) = 0\quad \text{for all }q\in Q\}.
\end{equation}

\subsection{The continuous problem}
Based on the above preliminaries, the variational formulation of problem \eqref{Brinkman equations} is: find $(\mathbf{u}, p)\in \mathbf{V}\times Q$ such that
 \begin{subequations}\label{continue weak formulation}
    \begin{align}
       \nu a(\mathbf{u}, \mathbf{v}) + b(\mathbf{v},p) = (\mathbf{f},\mathbf{v}) \qquad \forall\,\mathbf{v}\in\mathbf{V},	\label{continue weak formulation 1}		\\
        b(\mathbf{u},q) = 0 \qquad \forall\, q\in Q,	\label{continue form b}
    \end{align}
\end{subequations}	
where the continuous bilinear forms are defined as
$$
      a(\mathbf{u}, \mathbf{v}):=\int_{\Omega} \mathbf{\nabla} \mathbf{u}: \mathbf{\nabla} \mathbf{v} + \kappa^{-1} \mathbf{u}\cdot\mathbf{v} \,\mathrm{d\mathbf{x}}, \quad b(\mathbf{v}, q):=\int_{\Omega} q \nabla\cdot\mathbf{v}\,\mathrm{d\mathbf{x}},
$$
where the bilinear form $b(\cdot,\cdot)$ and the space $\mathbf{V}$ and $Q$ satisfy the inf-sup condition \cite{brezzi2012mixed}, i.e., $\exists \;\text{constant}\; \gamma$ such that
\begin{equation}	\label{continue LBB condition}
    \sup_{\mathbf{v}\in\mathbf{V},\mathbf{v}\neq 0} \frac{b(\mathbf{v}, q)}{\|\nabla\mathbf{v}\|} \ge \gamma \|q\|\qquad \forall\, q\in Q.
\end{equation}

It is obvious that the bilinear form $a(\cdot,\cdot)$ is continuous and  coercive: there exists a positive constant $C_a$ such that
\begin{subequations}
    \begin{align}
         &a(\mathbf{w},\mathbf{v}) \le C_a \interleave \mathbf{w} \interleave \cdot \interleave \mathbf{v} \interleave\qquad\forall\,\mathbf{w},\mathbf{v}\in\mathbf{V}.	\label{a continuity}		\\
         &a(\mathbf{v},\mathbf{v})\ge  \interleave \mathbf{v} \interleave^2\qquad \forall\,\mathbf{v}\in\mathbf{V}. \label{a coercity}
    \end{align}
\end{subequations}

Combining \eqref{continue LBB condition}, \eqref{a coercity} and the standard saddle point theory, \cite{wang2021Brinkman} one can get the well-posedness of the continue problem \eqref{continue weak formulation}.

\subsection{The Helmholtz decomposition}
For the load term $\mathbf{f}\in [L^2(\Omega)]^2$, there exists a unique Helmholtz decomposition \cite{John2017}
\begin{equation}	\label{Helmholtz}
    \mathbf{f} = \mathbb{P}(\mathbf{f}) + \nabla \alpha,
\end{equation}
where $\mathbb{P}(\mathbf{f})\in \left\{\mathbf{v}\in [L^2(\Omega)]^2:\;\nabla\cdot\mathbf{v}=0,\;\mathbf{v}=0\;\text{on}\; \partial\Omega\right\}$ is called the Helmholtz projector of $\mathbf{f}$ and $\alpha\in H^1(\Omega)$. Then by restricting  $\mathbf{v}\in\mathbf{Z}$, the weak form \eqref{continue weak formulation}  can be reformulated as the  divergence-free problem: find $\mathbf{u} \in \mathbf{Z}$ such that
\begin{equation}
    \nu a(\mathbf{u},\mathbf{v}) = \left(\mathbb{P}(\mathbf{f}),\mathbf{v}\right)\qquad \forall\,\mathbf{v}\in\mathbf{Z}.
\end{equation}
Furthermore, choosing $\mathbf{v} = \mathbf{u}$, combining \eqref{a coercity} and the Poincar\'{e} inequality with the constant $C_P$, we have
$$
 \nu\interleave\mathbf{u}\interleave^2 = \nu a(\mathbf{u},\mathbf{u})= \left(\mathbb{P}(\mathbf{f}),\mathbf{u}\right) \le C_P \|\mathbb{P}(\mathbf{f})\|\cdot |\mathbf{u}|_1\le C_P \|\mathbb{P}(\mathbf{f})\|\cdot \interleave\mathbf{u}\interleave,
$$
where the last step follows naturally, due to the fact that $\kappa$ is symmetric positive definite. Then we obtain the stability estimate for the velocity
\begin{equation}	\label{continuous bound of velocity}
    \interleave\mathbf{u}\interleave \le \frac{C_P}{\nu}\|\mathbb{P}(\mathbf{f})\|.
\end{equation}

\section{A really pressure-robust Virtual element}	\label{section space}
In this section, we introduce the concept of polygonal subdivision of $\Omega$ and the basic settings of the VEM. Subsequently, we constructed a divergence-preserving reconstruction operator based on the lowest-order Raviart-Thomas finite element space $\mathrm{RT}_0$, and applied it to the discretization of the right-hand side, thereby proposing a really pressure-robust and well-posed VEM scheme.

\subsection{Projection operators and virtual element spaces}
 Let $\{\Omega_h\}$ represent the sequence of $\Omega$ subdivided into a set of general polygonal elements $E$ with $h_E :=\text{diameter}(E)$,  mesh size $h:=\max_{E\in\Omega_h} h_E$ and boundary $\partial E$. Moreover, $e$ stands for a generic edge of an element $E$; $N^{V}_E$ denotes the number of vertices in $E$, and $V_i$, $1\le i\le N^V_E$ represent any vertex in $E$. We also denote the outward unit normal vector and unit tangential vector on $\partial E$ by $\mathbf{n}_E$ and $\mathbf{t}_E$, respectively. The subscript E will be dropped in a clear space. Moreover, denote by $\mathcal{E}_e$ the set of elements taking $e$ as an edge, and let $\mathcal{F}^0_h$ be a set of interior edge shared by two elements. We suppose that  $\{\Omega_h\}$ satisfy the following assumptions (refer \cite{veiga2013, veiga2014, Verma2021, xiong2024}):
\begin{Assumption}\label{assumption regularity}
    There exists a positive real number $\rho$ such that for all $h$ and for every $E\in\Omega_h$: (1) the ratio between the shortest edge $e_{\min}$ and the diameter $h_E$ of $E$ is greater than $\rho$, i.e., $\frac{e_{\min}}{h_E}>\rho$; (2) $E$ is star-shaped with respect to a ball of radius $\rho h_E$ and center $\mathbf{x}_E\in E$.
\end{Assumption}

For any non-negative integer $k$, $P_k(\mathcal{D})$ denotes the space of polynomials with total degree less than or equal to $k$ on the  bounded domain $\mathcal{D}$ of dimension $d\;(d = 1,2)$, as an edge or an element, and the corresponding vector polynomial space  denotes as $\mathbf{P}_k(\mathcal{D})$. Let $\mathbf{x}_{\mathcal{D}}$ and $h_{\mathcal{D}}$ be denoted as the barycenter and the diameter of $\mathcal{D}$, respectively. We denote by $\mathcal{M}_k(\mathcal{D})$ the scaled monomial set as a basis for the polynomial space $\mathbf{P}_k({\mathcal{D}})$ defined on $\mathcal{D}$ \cite{Ahmad2013,veiga2013,veiga2014}. The set of scaled monomials $\mathcal{M}_k(\mathcal{D})$ with degrees less than or equal to $k$  is defined as follows
$$
\mathcal{M}_k(\mathcal{D}) := \left\{m\; :\; m = \left( \frac{\mathbf{x}-\mathbf{x}_{\mathcal{D}}}{h_\mathcal{D}}\right)^{\alpha}\quad \text{for}\;\; \alpha\in\mathbb{N}^d, \; 0\le |\alpha|\le k\right\}.
$$

For simplicity, we also define the following spaces
\vspace{0.5em}
\begin{itemize}
    \setlength{\itemsep}{1em}
    \item  $\mathcal{G}_k(E):=\nabla P_{k+1}(E) \subseteq \mathbf{P}_k(E)$;
    \item  $\mathcal{G}_k^{\perp}(E):=\mathbf{x}^{\perp} P_{k-1}(E) \subseteq \mathbf{P}_k(E)\;$ with $\;\mathbf{x}^{\perp}=(y,-x)^{\mathrm{T}}$;
    \item  $\mathbf{B}_1(\partial E):=\left\{\mathbf{v} \in \mathbf{C}^0(\partial E),\left.\mathbf{v}\right|_e \cdot \mathbf{n}_E \in P_2(e),\left.\mathbf{v}\right|_e \cdot \mathbf{t}_E \in P_1(e) \quad{{\forall}} e \in \partial E\right\}$.
\end{itemize}
 \vspace{0.5em}
Following \cite{Xiong2024b,Verma2021, Wanggang2021}, we introduce the local lowest order virtual element space for the velocity as follows:
$$
\begin{aligned}
    \mathbf{U}_E:= & \left\{\mathbf{v} \in \mathbf{H}^1(E):\left.\mathbf{v}\right|_{\partial E} \in \mathbf{B}_1(\partial E), \Delta \mathbf{v}-\nabla s=\mathbf{0}\; \text { and } \;\nabla \cdot \mathbf{v} \in P_0(E)\right. \\
    & \text { for some } \left.s \in L_0^2(E)\right\} .
\end{aligned}
$$
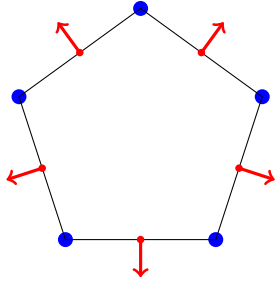
\begin{figure}[ht]
    \centering  
        \begin{tikzpicture}
            \pgfmathsetmacro{\side}{1.7} 
            \pgfmathsetmacro{\angle}{360/5} 
            
            \foreach \i in {1,...,5} {
                \coordinate (P\i) at ({\angle/4 + (\i-1)*\angle}:\side);
                \node at (P\i) [circle, fill=blue, inner sep=2pt] {};
            }
            
            \draw (P1) -- (P2) -- (P3) -- (P4) -- (P5) -- cycle;
            
             \foreach \i in {1,...,5} {
                 \pgfmathtruncatemacro{\j}{mod(\i, 5) + 1}
                 \coordinate (M\i) at ($(P\i)!0.5!(P\j)$);
                 \coordinate (N\i) at ($ (M\i)!0.5cm!-90:(P\j) $); 
                 \draw[red, ->, thick, line width=1.2pt] (M\i) -- (N\i);
                 \node at (M\i) [circle, fill=red, inner sep=1pt] {};
             }
        \end{tikzpicture}
    \caption{Illustration of degrees of freedom. We represent $\mathbf{D}_{\mathbf{V}}1$ with the blue dots, $\mathbf{D}_{\mathbf{V}}2$ with the red arrows.  }
    \label{DoFs}
\end{figure}

Correspondingly, as illustrated in Fig. \ref{DoFs}, the degrees of freedom $\mathbf{D}_{\mathbf{V}}$ for local space $\mathbf{U}_E$  include the following:
\begin{subequations}	\label{degrees of freedom}
    \begin{align}
             &\text{$\mathbf{D}_{\mathbf{V}}1$: the values of $\mathbf{v}$ at the vertices of the polygon $E$},	\label{dof1}		\\
             &\text{$\mathbf{D}_{\mathbf{V}}2$: $\frac{1}{|e|} \int_e \mathbf{v} \cdot \mathbf{n} \mathrm{d} s \quad \forall e \in \partial E$.}	\label{dof2}
    \end{align}
\end{subequations}

Prior to defining the global virtual element spaces, we must first define two essential projection operators that ensure the computability of discrete linear forms. For any $\mathbf{v} \in \mathbf{U}_E$, an energy projection $\Pi^{\nabla, E}$ onto $\mathbf{P}_1(E)$ is defined by
\begin{subequations}
    \begin{align}
        \left(\nabla\left(\mathbf{v}-\Pi^{\nabla, E} \mathbf{v}\right), \nabla \mathbf{p}_1\right)_E & =0 \quad \forall\, \mathbf{p}_1 \in \mathbf{P}_1(E), 	\label{energy projection 1}	\\
        P^{0,E} (\Pi^{\nabla, E}\mathbf{v} - \mathbf{v}) &= 0,
    \end{align}
\end{subequations}
where the projector $P^{0,E} \mathbf{v}:=\frac{1}{N^V_E} \sum^{N^V_E}_{i=1}\mathbf{v}(V_i)$ is  onto the space of constants. It is obvious that $\Pi^{\nabla, E}\mathbf{p}_1 = \mathbf{p}_1$ holds for any $\mathbf{p}_1 \in \mathbf{P}_1(E)$. 
\vspace{0.5em}
\begin{remark}
   \rm{ (The computability of $\Pi^{\nabla, E}$) By integrating by parts we can know that the term $(\nabla \mathbf{v}, \nabla \mathbf{p}_1)_E = (\mathbf{v}, \nabla \mathbf{p}_1 \cdot\mathbf{n}_E)_{\partial E}$, where $\nabla \mathbf{p}_1\cdot \mathbf{n}_E$ is obviously a constant vector. And there exits a decomposition $\mathbf{v} = (\mathbf{v}\cdot\mathbf{n}_E)\mathbf{n}_E + (\mathbf{v}\cdot\mathbf{t}_E)\mathbf{t}_E$ on boundary $\partial E$. On the one hand, since $\mathbf{v}\cdot\mathbf{n}_E \in P_2(e)$ for $e\in\partial E$, we can easily calculate its integral on each edge, relying solely on the degrees of freedom \eqref{dof1}-\eqref{dof2}. On the other hand, since $\mathbf{v}\cdot\mathbf{t}_E \in P_1(e)$ for $e\in\partial E$, we can easily calculate its integral on each edge by the trapezoidal rule and the vertex degrees of freedom \eqref{dof1}. Then the integral in \eqref{energy projection 1} is ordinary.}
\end{remark}
\vspace{0.5em}

We proceed to define a local $\mathbf{L}^2$-projection operator $\Pi^{0, E}$ from $\mathbf{U}_E $ onto $ \mathbf{P}_1(E)$, which is defined by
\begin{equation}		\label{L2 projection}
    \left(\Pi^{0, E}  \mathbf{v} - \mathbf{v}, \mathbf{p}_1\right)_E= 0 \quad \forall \mathbf{p}_1 \in \mathbf{P}_1(E),
\end{equation}
Clearly, the two projection operators possess the subsequent estimations
\begin{subequations}
   \begin{align}
        \|\Pi^{0,E} \mathbf{v}\|_{E} &\le \|\mathbf{v}\|_E,  \label{Pi0 bound} \\
        |\Pi^{\nabla,E} \mathbf{v}|_{1,E} &\le |\mathbf{v}|_{1, E}. 	\label{Pis bound}
   \end{align}
\end{subequations}

It is evident that the term $(\mathbf{v}, \mathbf{p_1})_E$ is uncomputable for any $\mathbf{v}\in \mathbf{U}_E$. This observation has prompted us to define a modified space in which $(\mathbf{v}, \mathbf{p}_1)_E$ becomes computable. With the help of enhanced technology \cite{Ahmad2013}, we first enlarge the space $\mathbf{U}_E$ to be
$$
\begin{aligned}
    \widehat{\mathbf{U}}_E:= & \left\{\mathbf{v} \in \mathbf{H}^1(E):\left.\mathbf{v}\right|_{\partial E} \in \mathbf{B}_1(\partial E), \Delta \mathbf{v}-\nabla s \in \mathcal{G}_1^{\perp}(E) \;\text { and } \;\nabla \cdot \mathbf{v} \in P_0(E)\right. \\
    & \text { for some } \left.s \in L^2(E)\right\} .
\end{aligned}
$$
We subsequently define the local virtual element space $\mathbf{V}_E$ as the restriction of $\widehat{\mathbf{U}}_E$ given by
\begin{equation}
   \mathbf{V}_E:=\left\{\mathbf{v} \in \widehat{\mathbf{U}}_E:\left(\mathbf{v}-\Pi^{\nabla} _E \mathbf{v}, \mathbf{p}_1\right)_E=0 \text { for any } \mathbf{p}_1 \in \mathcal{G}_1^{\perp}(E)\right\}.
\end{equation}
As mentioned  in \cite{Ahmad2013, Wanggang2023, xiong2024}, the dimension of  $\mathbf{V}_E$ is equal to that of original space $\mathbf{U}_E$, and these two spaces share the same degrees of freedom \eqref{DoFs}.
\vspace{0.5em}
\begin{remark}
  \rm{(The computability of $\Pi^{0, E}$) For any $\mathbf{p}_1 \in \mathbf{P}_1(E)$, there exists a unique decomposition  $\mathbf{p}_1 = \nabla s + \mathbf{g}$, where $s\in P_2(E)$ and $\mathbf{g}\in \mathcal{G}_1^{\perp}(E)$.  Furthermore, the constant $\nabla \cdot \mathbf{v} \in P_0(E)$ can be readily computed using  the divergence theorem along with the normal degree of freedom \eqref{dof2}. As demonstrated in the work \cite{Wanggang2021},  it is straightforward to explicitly express $\mathbf{v}|_e \cdot\mathbf{n}$ and $\mathbf{v}|_e \cdot\mathbf{t}$ on edge $e$, relying only on the degrees of freedom \eqref{dof1}-\eqref{dof2}. It follows from  integration by parts and the definition of $\Pi^{\nabla, E}$ that
   $$
   (\mathbf{v}, \mathbf{p}_1)_E = (\mathbf{v}, \nabla s + \mathbf{g})_E = -(\nabla\cdot \mathbf{v}, s)_E + (\mathbf{v}\cdot\mathbf{n}, s)_{\partial E} + (\Pi^{\nabla, E}\mathbf{v}, \mathbf{g})_E,
   $$
  whose right-hand side is evidently computable by using   the computable $\nabla\cdot\mathbf{v}$ and $\Pi^{\nabla, E}\mathbf{v}$ within $E$ and $\mathbf{v}|_{\partial E}\cdot \mathbf{n}$  on $\partial E$. Thereby the integral in \eqref{L2 projection} is ordinary.}
\end{remark}
\vspace{0.5em}
Finally, the global virtual element space for velocity is constructed by assembling the local spaces $\mathbf{V}_E$ as follows:
$$
\mathbf{V}_h:=\left\{\mathbf{v} \in \mathbf{H}_0^1(\Omega): \mathbf{v} \in \mathbf{V}_E\quad \forall E \in \mathcal{T}_h\right\}.
$$

Moreover, the approximated space $Q_h$ for the pressure  is defined as the piecewise constant space, denoted by
\begin{equation}
    Q_h := \left\{q\in L^2_0(\Omega):\; q\in {P}_{0}(E)\quad\forall E\in \mathcal{T}_h\right\}.
\end{equation}

We define the exact discrete kernel space $\mathbf{Z}_h$ by utilizing the fact that $\nabla\cdot \mathbf{V}_h \subset Q_h$, as follows:
\begin{equation}	\label{discrete kernel space}
    \mathbf{Z}_h:=\left\{\mathbf{v}_h \in \mathbf{V}_h: b\left(\mathbf{v}_h, q_h\right)=0 \quad\forall q_h \in Q_h\right\}=\left\{\mathbf{v}_h \in \mathbf{V}_h: \nabla\cdot \mathbf{v}_h=0\right\}.
\end{equation}

\subsection{Divergence-preserving reconstruction operator}
For each element $E\in \Omega_h$, we assume the existence of a subtriangulation that subdivides $E$ into $n_t$ regular triangles, with $n_t$ is a constant correspond to  $E$. A constrained Delaunay triangulation of the polygon yields a conforming simplicial submesh $\Omega_h^\star$ relative to $\Omega_h$, with each triangle $T \subset \Omega_h^\star$ uniquely corresponding to an element $E \in \Omega_h$ ($T \subset E$), and each $E \in \Omega_h$ composed of several triangles from $\Omega_h^\star$. Consider polygon \( E \) with vertices \( V_i \) (for \( i = 1, \ldots, N^V_E \)), ordered counterclockwise. Let \( \mathbf{n}_i^{\partial} \) denote the outward unit normal on edge \( e_i^{\partial} \) from \( V_i \) to \( V_{i+1} \). Using modulus \( N^V_E \) indexing, \( E \) is triangulated into \( N^V_E - 2 \) triangles \( \{T_i\}_{i=1}^{N^V_E-2} \) without extra vertices or edges. The \( N^V_E - 3 \) inner edges \( e_j^{o} \) (for \( j = 1, \ldots, N^V_E - 3 \)) have fixed normal $ \mathbf{n}_j^{0} $. The above notations are illustrated in Fig. \ref{subtriangulation}. 
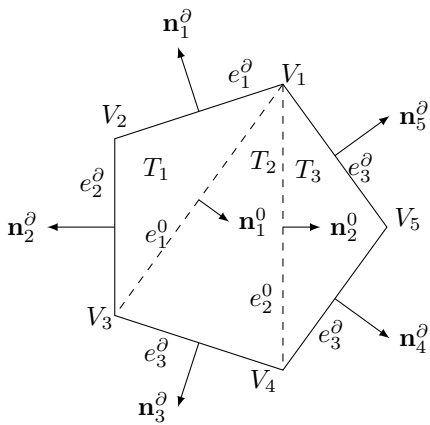
\begin{figure}	
    \centering 
    \begin{tikzpicture}
        \coordinate (A) at (72:2cm);
        \coordinate (B) at (144:2cm);
        \coordinate (C) at (216:2cm);
        \coordinate (D) at (288:2cm);
        \coordinate (E) at (360:2cm); 
        \draw (A) -- (B) -- (C) -- (D) -- (E) -- cycle;    
        \node at ($(A)+(40:0.2cm)$) {$V_1$};
        \node at ($(B)+(90:0.3cm)$) {$V_2$};
        \node at ($(C)+(190:0.2cm)$) {$V_3$};
        \node at ($(D)+(210:0.3cm)$) {$V_4$};
        \node at ($(E)+(380:0.3cm)$) {$V_5$};
        
       \draw[dashed] (A) -- (C);
       \draw[dashed] (A) -- (D);
       \coordinate (M1) at ($(A)!0.5!(B)$);  
       \coordinate (N1) at ($(M1)!0.9cm!-90:(B)$);
       \draw [-latex] (M1)--(N1) node[above] {$\mathbf{n}_1^{\partial}$};
       \draw ($(M1)!0.5!(A)$) node[above] {$e^\partial_1$};   
       \coordinate (M2) at ($(B)!0.5!(C)$);  
       \coordinate (N2) at ($(M2)!0.9cm!-90:(C)$);
       \draw [-latex] (M2)--(N2) node[left] {$\mathbf{n}_2^{\partial}$};
       \draw ($(M2)!0.5!(B)$) node[left] {$e^\partial_2$};
       
       \coordinate (M3) at ($(C)!0.5!(D)$);  
       \coordinate (N3) at ($(M3)!0.9cm!-90:(D)$);
       \draw [-latex] (M3)--(N3) node[left] {$\mathbf{n}_3^{\partial}$};
        \draw ($(M3)!0.5!(C)$) node[below] {$e^\partial_3$};
        
       \coordinate (M4) at ($(D)!0.5!(E)$);  
       \coordinate (N4) at ($(M4)!0.9cm!-90:(E)$);
       \draw [-latex] (M4)--(N4) node[right] {$\mathbf{n}_4^{\partial}$};
       \draw ($(M4)!0.5!(D)$) node[right] {$e^\partial_3$};
       
       \coordinate (M5) at ($(E)!0.5!(A)$);  
       \coordinate (N5) at ($(M5)!0.9cm!-90:(A)$);
       \draw [-latex] (M5)--(N5) node[right] {$\mathbf{n}_5^{\partial}$};
       \draw ($(M5)!0.5!(E)$) node[above] {$e^\partial_3$};
       
       \coordinate (M01) at ($(A)!0.5!(C)$);  
       \coordinate (N01) at ($(M01)!0.5cm!-90:(A)$);
       \draw [-latex] (M01)--(N01) node[right] {$\mathbf{n}_1^{0}$};
       \draw ($(M01)!0.5!(C)$) node[above] {$e^0_1$};
       
       \coordinate (M02) at ($(A)!0.5!(D)$);  
       \coordinate (N02) at ($(M02)!0.5cm!-90:(A)$);
       \draw [-latex] (M02)--(N02) node[right] {$\mathbf{n}_2^{0}$};
       \draw ($(M02)!0.5!(D)$) node[left] {$e^0_2$};
       
       \draw ($(B)!0.5!(M01)$) node {$T_1$};
       \draw ($(A)!0.5!(M01)$) node[below right] {$T_2$};
       \draw ($(M5)!0.5!(M02)$) node[above] {$T_3$};
       
    \end{tikzpicture}
    \caption{A subtriangulation of a pentagon $E$ and the corresponding notations.}  \label{subtriangulation}
\end{figure}

A local space $\mathcal{A}_E$ \cite{Wanggang2023} as 
$$
\mathcal{A}_E = \left\{\mathbf{v}\in \mathbf{H}(\dive, E):\; \mathbf{v}|_{T_i} \in \mathrm{RT}_0(T_i), \;\nabla\cdot \mathbf{v}\in P_0(E)\quad \forall \,T_i \in E\right\}, 
$$
where $\mathrm{RT}_0(T_i)$ denotes the usual lowest-order Raviart-Thomas finite element space in simplex $T_i$. Naturally, the  global space $\mathcal{A}_h$ is defined as 
$$
\mathcal{A}_h = \left\{\mathbf{v}\in \mathbf{H}(\dive, \Omega):\; \mathbf{v}|_E \in \mathcal{R}_E\quad \forall\,E\in\Omega_h\right\}.
$$ 
Then we define the divergence-preserving reconstruction operator $\mathcal{R}_h:\; \mathbf{V} \rightarrow \mathcal{A}_h$ that satisfies a set of conditions as follows
\begin{subequations}
    \begin{align}
        &\int_{e^{\partial}_i} \mathcal{R}_h \mathbf{v} \cdot \mathbf{n}_i^{\partial}\;\mathrm{ds} = \int_{e^{\partial}_i}  \mathbf{v} \cdot \mathbf{n}_i^{\partial}\;\mathrm{ds}\qquad \forall \; \text{boundary edge} \;e^{\partial}_i \in \partial E, 	\label{Rh 1}    	 \\
        &\int_{e^{0}_i} [\mathcal{R}_h \mathbf{v}] \cdot \mathbf{n}_i^{0}\;\mathrm{ds} = 0 \qquad \forall \; \text{inner edge} \;e^{0}_i \in E,	\label{Rh 2}		\\
        &\int_{T_1} \nabla\cdot \left(\mathcal{R}_h\mathbf{v}|_{T_i} - \mathcal{R}_h\mathbf{v}|_{T_1}\right) \mathrm{d\mathbf{x}} = 0\qquad \text{for}\;i =2,\cdots, N^V_E -2. 	\label{Rh 3}
    \end{align}
\end{subequations}
In \eqref{Rh 2}, the notation $[\cdot]$ denote the jump of function on an edge. It should be emphasized that the reconstruction operator $\mathcal{R}_h$ is explicitly computable only rely on  the basis function of the lowest-order Raviart-Thomas element and degrees of freedom \eqref{degrees of freedom} of virtual element. One can see the Appendix of \cite{Wanggang2023} for the implementation details of the reconstruction operator on a polygon. 

The reconstruction operator $\mathcal{R}_h$ satisfies the following properties \cite{Wanggang2023,Ye2021}.

\begin{lema}
    For all $\mathbf{v}\in \mathbf{V} $, the reconstruction operator has the following properties
    \begin{subequations}
        \begin{align}
            &\int_e \mathcal{R}_h \mathbf{v}\cdot \mathbf{n} \;\mathrm{ds} = \int_e \mathbf{v}\cdot\mathbf{n}\;\mathrm{ds}\qquad \forall \; \text{edge $e$ of } E,\\
            &\int_E \nabla\cdot\mathcal{R}_h \mathbf{v}\;\mathrm{d\mathbf{x}} = \int_E \nabla\cdot\mathbf{v}\;\mathrm{d\mathbf{x}}\qquad \forall\, E\in\Omega_h,	\label{reconstruction 2}\\
            &\|\mathcal{R}_h\mathbf{v} - \mathbf{v}\|\le C_{\mathcal{R}_1} h |\mathbf{v}|_1,	 \label{reconstruction 3}
        \end{align}
    \end{subequations}
    where $C_{\mathcal{R}_1}$ is a positive constant independent of $h$.
\end{lema}

Moreover, we can derive the follow result from  \eqref{reconstruction 3} and inverse inequality  in \cite{Chenlong2018some}:
\begin{equation}
    |\mathcal{R}_h\mathbf{v} - \mathbf{v}|_1 \le C_{\mathcal{R}_2} |\mathbf{v}|_1,	 \label{reconstruction 4 for h1 estimate}
\end{equation}
where $C_{\mathcal{R}_2}$ is a positive constant independent of $h$.

\subsection{Discrete scheme}	\label{discrete linear form}
For all $\mathbf{w}_h, \mathbf{v}_h \in \mathbf{V}_E$, we first define the computable local discrete linear form by
$$
\begin{aligned}
    a_h^E\left(\mathbf{w}_h, \mathbf{v}_h\right)= &\left(\Pi^{0, E} \nabla \mathbf{w}_h, \Pi^{0, E} \nabla \mathbf{v}_h\right)+S^{\nabla}_E\left(\left(\mathbf{I}-\Pi^{\nabla, E}\right) \mathbf{w}_h,\left(\mathbf{I}-\Pi^{\nabla, E}\right) \mathbf{v}_h\right), \\
    & + \left(\kappa^{-1}\Pi^{0, E} \mathbf{w}_h, \Pi^{0, E}  \mathbf{v}_h\right) + S^{0}_E\left(\left(\mathbf{I}-\Pi^{0, E}\right) \mathbf{w}_h,\left(\mathbf{I}-\Pi^{0, E}\right) \mathbf{v}_h\right),
\end{aligned}
$$

where the symmetric stabilizing bilinear forms $S^{\nabla}_E(\cdot, \cdot)$ and $ S^{0}_E(\cdot, \cdot)$  satisfies
\begin{subequations}	\label{stabilzation eq}
    \begin{align}
        C_1^\nabla \|\nabla \mathbf{v}_h\|_E^2 \le S^{\nabla}_E\left(\mathbf{v}_h, \mathbf{v}_h\right)\le C_2^\nabla \|\nabla \mathbf{v}_h\|_E^2   \quad \forall \,\mathbf{v}_h \in \mathbf{V}_E \;\text{and} \; \Pi^\nabla_E \mathbf{v}_h = 0, \\
        C_1^0 \|\kappa^{-1/2} \mathbf{v}_h\|_E^2 \le S^{0}_E\left(\mathbf{v}_h, \mathbf{v}_h\right)\le C_2^0 \|\kappa^{-1/2}\mathbf{v}_h\|_E^2   \quad \forall \,\mathbf{v}_h \in \mathbf{V}_E \;\text{and} \; \Pi^0_E \mathbf{v}_h = 0,
    \end{align}
\end{subequations}
for some positive constants $C_1^\nabla, C_2^\nabla, C_1^0 $ and $C_2^0$ independent of $h_E$. Furthermore, one can easily verify $a^E_h(\cdot,\cdot)$ satisfies the consistency  and the stability properties
\begin{subequations}
    \begin{align}
            a^E_h(\mathbf{q},\mathbf{v}_h) = a^E(\mathbf{q},\mathbf{v}_h)\qquad \forall\,\mathbf{q} \in [\mathbb{P}_1(E)]^2,\,\mathbf{v}_h\in\mathbf{V}_h, 	\label{consistency ah}		\\
            \alpha_* a^E(\mathbf{v}_h, \mathbf{v}_h) \le a^E_h(\mathbf{v}_h, \mathbf{v}_h) \le \alpha^* a^E(\mathbf{v}_h, \mathbf{v}_h)\qquad \forall\,\mathbf{v}_h\in\mathbf{V}_h, \label{stability ah}
    \end{align}
\end{subequations}
where two positive constants $\alpha_*:=\min\{1, C^\nabla_1, C^0_1\}$ and $\alpha^*:=\max\{1,C^\nabla_2, C^0_2\}$ are independent of $h_E$. 

Naturally, $a_h$ is a global extension of $a_E$, i.e., 
$$
a_h(\mathbf{w}_h, \mathbf{v}_h) = \sum_{E \in \Omega_h} a^E_h(\mathbf{w}_h, \mathbf{v}_h)\qquad \forall\,\mathbf{w}_h, \mathbf{v}_h \in \mathbf{V}_h.
$$
The coercivity and continuity of $a_h(\cdot,\cdot)$ can be immediately obtained by the equivalent properties \eqref{stabilzation eq}, Cauchy-Schwartz inequality and the definition of energy norm $\interleave\cdot\interleave$ in \eqref{energy norm}, i.e., there exits two positive constant $C^*$ and  $C_* (= \alpha_*)$ such that
\begin{subequations}	\label{ah properties}
    \begin{align}
        &a_h(\mathbf{w}_h, \mathbf{v}_h) \le C^* \interleave\mathbf{w}_h\interleave\,\interleave\mathbf{v}_h\interleave\qquad \forall\,\mathbf{w}_h, \mathbf{v}_h \in \mathbf{V}_h,	\label{continuity ah}\\
        &a_h(\mathbf{v}_h, \mathbf{v}_h) \ge C_*\interleave\mathbf{v}_h\interleave^2\qquad \forall\, \mathbf{v}_h \in \mathbf{V}_h.	\label{coercivity ah}
    \end{align}
\end{subequations}
Following the same approach of Lemma 4.3 in \cite{veiga2010}, the discrete inf-sup condition for the bilinear form $d(\cdot,\cdot)$ holds true, i.e., $\exists \;\text{constant}\; \widetilde{\gamma}>0$ such that 
\begin{equation}	\label{discrete inf-sup}
    \sup_{\mathbf{v}_h\in \mathbf{V}_h/ {\{\mathbf{0}\}}} \frac{d(\mathbf{v}_h, q_h)}{\interleave \mathbf{v}_h\interleave} \ge \widetilde{\gamma} \|q_h\|\qquad \forall\,q_h \in Q_h.
\end{equation}

Moreover, utilizing the projection property \eqref{reconstruction 3} and Friedrichs inequality \eqref{Friedrichs}, we can deduce that for any $\mathbf{v}_h\in \mathbf{V}_h$, the following holds:
\begin{equation}	\label{Rh vh}
    \begin{aligned}
        \|\mathcal{R}_h \mathbf{v}_h\| =&\; \|\left(\mathcal{R}_h \mathbf{v}_h - \mathbf{v}_h + \mathbf{v}_h\right)\| \le (C_{\mathcal{R}_1} h + 1) \|\mathbf{v}_h\|_1 \\
        \le& \;C_F (C_{\mathcal{R}_1} h + 1) \|\nabla\mathbf{v}_h\|\\
        \le& \;\gamma^\star \|\nabla\mathbf{v}_h\|,
    \end{aligned}
\end{equation}
where we have used the boudedness of the mesh size at the last step, therefore, $\gamma^\star$ is a positive constant independent of $\kappa$ and $h$. 

Then we  construct a really pressure-robust virtual element scheme for Brinkman problem \eqref{Brinkman equations}: Find $(\mathbf{u}_h, p_h)\in\mathbf{V}_h\times Q_h$
\begin{subequations}	\label{discrete scheme}
    \begin{align}
        \nu a_h(\mathbf{u}_h, \mathbf{v}_h) + b(\mathbf{v}_h, p_h) = \left(\mathbf{f}, \mathcal{R}_h\mathbf{v}_h\right)\qquad \forall\,\mathbf{v}_h\in\mathbf{V}_h,	\label{discrete scheme 1}		\\
        b(\mathbf{u}_h, q_h) =0\qquad\forall\, q_h\in Q_h.	\label{discrete scheme 2}
    \end{align}
\end{subequations}
 Combining the properties \eqref{ah properties}-\eqref{discrete inf-sup}, one can get an immediate consequence that there exists a unique solution $\left(\mathbf{u}_h, p_h\right)\in\mathbf{V}_h\times Q_h$ of scheme \eqref{discrete scheme}. Moreover, since $\nabla\cdot \mathbf{V}_h \subset Q_h$,  we can easily deduce  that $\nabla\cdot \mathbf{u}_h =0$ from \eqref{discrete scheme 2}, which  indicates that the scheme \eqref{discrete scheme} is exactly divergence-free, ensuring the mass conservation of the system. Furthermore, by utilizing the discrete kernel space $\mathbf{Z}_h$ defined in \eqref{discrete kernel space}, the problem \eqref{discrete scheme} can be formulated into the equivalent kernel form:
\begin{equation}
     \nu a_h(\mathbf{u}_h, \mathbf{v}_h)  = \left(\mathbb{P}(\mathbf{f}), \mathcal{R}_h\mathbf{v}_h\right)\qquad \forall\,\mathbf{v}_h\in\mathbf{Z}_h,	\label{discrete kernel scheme}
\end{equation}
where we have used the fact that $\nabla\cdot\mathcal{R}_h \mathbf{v}_h =0$ for any $\mathbf{v}\in \mathbf{Z}_h$ from \eqref{reconstruction 2} and the Helmholtz decomposition \eqref{Helmholtz}. Moreover, we define the notations $(\Pi^{0,h}\cdot)|_E := \Pi^{0,E}\cdot$ and  $(\Pi^{\nabla,h}\cdot)|_E := \Pi^{\nabla,E}\cdot$, and give the  following remark.
\begin{remark}
    Corresponding to $\left(\mathbf{f}, \mathcal{R}_h\mathbf{v}_h\right)$ in  the scheme \eqref{discrete scheme},  the right-hand side term in the standard VEM scheme is constructed as $(\mathbf{f}, \Pi^{0,h}\mathbf{v}_h)$. As discussed in \cite{Frerichs2020}, the projection operator $\Pi^{0,h}$ changes the divergence, thus destroying the orthogonality between the divergence-free function and the gradient force. This leads to the occurrence of the locking phenomenon as $\nu\rightarrow 0$. Fortunately, we have constructed a divergence-preserving reconstruction operator $\mathcal{R}_h$ such that $\nabla\cdot\mathcal{R}_h \mathbf{v}_h =0$ for any $\mathbf{v}_h\in \mathbf{Z}_h$. Thus we call \eqref{discrete scheme} as  a really pressure-robust virtual element scheme, and the fact will also be proved that prior error estimate of velocity is independent of continue pressure $p$ and viscosity $\nu$ in  subsequent analysis.
\end{remark}

Similar to \eqref{continuous bound of velocity}, we choosing $\mathbf{v}_h = \mathbf{u}$ from \eqref{discrete kernel scheme}, recalling the coercivity property \eqref{coercivity ah} and \eqref{Rh vh}, one can deduce the 
$$
\begin{aligned}
    \nu C_* \interleave\mathbf{u}_h\interleave^2 \le &\;\nu a_h(\mathbf{u}_h,\mathbf{u}_h) \le \|\mathbb{P}(\mathbf{f})\|\cdot \|\mathcal{R}_h\mathbf{u}_h\|\\
    \le & \;\gamma^\star \|\mathbb{P}(\mathbf{f})\|\cdot \|\nabla\mathbf{u}_h\|\\
    \le &\; \gamma^\star \|\mathbb{P}(\mathbf{f})\|\cdot \interleave\mathbf{u}_h\interleave.
\end{aligned}
$$
After eliminating $\interleave\mathbf{u}_h\interleave$ from both sides, one can immediately deduce the upper bound for the velocity $\mathbf{u}_h$:
\begin{equation}
    \interleave\mathbf{u}_h\interleave \le \frac{\gamma^\star}{\nu C_*} \|\mathbb{P}(\mathbf{f})\|.
\end{equation}

\section{A priori error analysis}
In this section, we derive the optimal error estimates for the velocity error in the energy norm and the pressure error in the $L^2$ norm. Firstly, we provide the following classical interpolation estimates holds for the spaces $\mathbf{V}_h$ and $Q_h$. 
\begin{lema}	\label{lemma calssical interpolation}
    Let $(\mathbf{v}, p) \in\left([H^{2}(\Omega)]^2 \cap \mathbf{V}\right) \times\left(H^1 (\Omega) \cap Q\right)$. Then there exists an approximation $\left(\mathbf{v}_I, p_I\right) \in \mathbf{V}_h \times Q_h$ such that
    $$
    \begin{aligned}
        \left\|\mathbf{v}-\mathbf{v}_I\right\|+h\left|\mathbf{v}-\mathbf{v}_I\right|_{1} & \leq C_{I_1} h^{2}\|\mathbf{v}\|_{2}, \\
        \left\|p-p_I\right\| & \leq C_{I_2} h\|p\|_1,
    \end{aligned}
    $$
    where the positive constants $C_{I_1}$ and $C_{I_2}$ are independent of $h$.
\end{lema}

Then we state the following classical approximation result for $\mathbf{P}_k(E)$ space on local star-shaped domains:
\begin{lema}		\label{lemma calssical polynomial estimate}
    Let $E \in \Omega_h$, and let two real numbers $s, p$ with $0 \leq s \leq 1$ and $1 \leq p \leq \infty$. Then for all $\mathbf{u} \in\left[H^{s+1}(\Omega)\right]^2$, there exists a polynomial function $\mathbf{u}_\pi \in$ $\mathbf{P}_k(E)$, such that
    $$
    \left\|\mathbf{u}-\mathbf{u}_\pi\right\|_{L^p(E)}+h_E\left|\mathbf{u}-\mathbf{u}_\pi\right|_{W^{1, p}(E)} \leq C_{\pi} h_E^{s+1}|\mathbf{u}|_{W^{s+1, p}(E)},
    $$ 
    where the positive constant $C_\pi$ independent of $h$.
\end{lema}

 Based on the above preparations, the optimal a priori error estimate of the pressure-robust VEM scheme \eqref{discrete scheme} is given by the following theorem.
 
 \begin{thm}	\label{thm prior u1 and p}
     Assume $(\mathbf{u}, p) \in\left([H^{2}(\Omega)]^2 \cap \mathbf{V}\right) \times\left(H^1 (\Omega) \cap Q\right)$ and the Assumption \ref{assumption regularity} holds, let $(\mathbf{u},p)$ and $(\mathbf{u}_h,p_h)$ is solutions of the continue problem \eqref{continue weak formulation} and discrete problem \eqref{discrete scheme}, respectively. Then its satisfies the priori error estimates
     \begin{subequations}
         \begin{align}
             \interleave \mathbf{u} - \mathbf{u}_h\interleave \le \widehat{C} h, \label{priori estimate u}\\
             \|p-p_h\|\le \widetilde{C}h, \label{priori estimate p}
         \end{align}
     \end{subequations}
where the positive constant $\widehat{C}:= \frac{C_{\mathcal{R}_1}}{C_*} \left(\|\mathbf{u}\|_2 + \|\kappa^{1/2}\mathbf{u}\|\right) +  \left(\frac{C_{\mathcal{T}_2}}{C_*} +  C_{I_1} + C_{I_1}\|\kappa^{-1/2}\|_\infty\right) \|\mathbf{u}\|_2 $ is independent of $h$, $\nu$ and pressure $p$, and the $\widetilde{C}:= C_{I_2} \|p\|_1 + \nu\frac{ C_{s}}{\widetilde{\gamma}}$  independent of $h$.
 \end{thm}
\begin{proof}
    Let $\mathbf{u}_I \in \mathbf{V}_h$ be the interpolation of $\mathbf{u}$ and satisfy the estimate \eqref{lemma calssical interpolation}. For any $\mathbf{v}_h\in\mathbf{Z}_h$, it follows that
    \begin{equation}	\label{prior error eq uh}
        \begin{aligned}
            \nu a_h (\mathbf{u}_h - \mathbf{u}_I, \mathbf{v}_h) = &\;\nu a_h (\mathbf{u}_h, \mathbf{v}_h) - \nu a_h (\mathbf{u}_I, \mathbf{v}_h) \\
            =&\; \left(\mathbf{f}, \mathcal{R}_h\mathbf{v}_h\right)- \nu a_h (\mathbf{u}_I, \mathbf{v}_h)\\
            = &\;  \left[\left(\mathbf{f}, \mathcal{R}_h\mathbf{v}_h\right)- \nu a(\mathbf{u},\mathbf{v}_h)\right] + \left[\nu a(\mathbf{u},\mathbf{v}_h) -\nu a_h (\mathbf{u}_I, \mathbf{v}_h)\right]\\
            = &\; \mathcal{T}_1 + \mathcal{T}_2.
        \end{aligned}
    \end{equation}
    Then we derive the upper bounds for $\mathcal{T}_1$ and $\mathcal{T}_2$ item by item. For $\mathcal{T}_1$, it follows from \eqref{Brinkman eq 1} and \eqref{reconstruction 3} that
    \begin{equation}	\label{T1 eq}
        \begin{aligned}
            \mathcal{T}_1 =&\; \left( -\nu\Delta \mathbf{u} + \nu \kappa^{-1} \mathbf{u} + \nabla p, \mathcal{R}_h\mathbf{v}_h\right) - \nu a(\mathbf{u},\mathbf{v}_h)\\
             =&\; \nu a(\mathbf{u}, \mathcal{R}_h\mathbf{v}_h - \mathbf{v}_h)\\
             = &\; \nu(-\Delta \mathbf{u}, \mathcal{R}_h\mathbf{v}_h - \mathbf{v}_h) + \nu \left(\kappa^{-1}\mathbf{u}, \mathcal{R}_h\mathbf{v}_h - \mathbf{v}_h \right) \\
             \le &\; C_{\mathcal{R}_1} \nu h \left(\|\mathbf{u}\|_2 + \|\kappa^{1/2}\mathbf{u}\|\right) |\mathbf{v}_h|_1\\
             \le &\; C_{\mathcal{R}_1} \nu h \left(\|\mathbf{u}\|_2 + \|\kappa^{1/2}\mathbf{u}\|\right) \interleave\mathbf{v}_h\interleave.
        \end{aligned}
    \end{equation}
   For the second term, let $\mathbf{u}_\pi$ be the piecewise linear polynomial projection of $\mathbf{u}$. Based on the consistency property \eqref{consistency ah}, \eqref{a continuity}, \eqref{continuity ah}, combining Lemma \ref{lemma calssical interpolation} and Lemma \ref{lemma calssical polynomial estimate}, we have
   \begin{equation}		\label{T2 eq}
       \begin{aligned}
           \mathcal{T}_2 =& \;\nu a(\mathbf{u}-\mathbf{u}_\pi,\mathbf{v}_h)  + \nu a(\mathbf{u}_\pi,\mathbf{v}_h) -\nu a_h (\mathbf{u}_I, \mathbf{v}_h)\\
           =& \;  \nu a(\mathbf{u} - \mathbf{u}_\pi,\mathbf{v}_h) +\nu a_h (\mathbf{u}_\pi- \mathbf{u}_I, \mathbf{v}_h)\\
           \le &\; \nu \left(C_a \interleave\mathbf{u}-\mathbf{u}_\pi\interleave + C^*\interleave\mathbf{u}_\pi - \mathbf{u}_I \interleave\right)\interleave\mathbf{v}_h\interleave\\
           \le &\; \left(C_\pi C_a + C_\pi C^* + C_{I_1}C^*\right) \nu h \|\mathbf{u}\|_2 \left(1 + h\|\kappa^{-1/2}\|_\infty\right) \interleave\mathbf{v}_h\interleave\\
           \le &\; C_{\mathcal{T}_2} \nu h\|\mathbf{u}\|_2 \interleave\mathbf{v}_h\interleave,
       \end{aligned}
   \end{equation}
where we have used the boundedness of the mesh size at the last step, and the constant is defined as $C_{\mathcal{T}_2}:=\left(C_\pi C_a + C_\pi C^* + C_{I_1}C^*\right)\left(1 + \|\kappa^{-1/2}\|_\infty\right)$. Collecting the above \eqref{prior error eq uh}-\eqref{T2 eq}, taking $\mathbf{v}_h:=\mathbf{u}_h - \mathbf{u}_I$ and recalling the coercivity \eqref{coercivity ah}, ones can lead to the following bound:
$$
C_* \nu \interleave\mathbf{u}_h - \mathbf{u}_I\interleave \le C_{\mathcal{R}_1} \nu h \left(\|\mathbf{u}\|_2 + \|\kappa^{1/2}\mathbf{u}\|\right) + C_{\mathcal{T}_2} \nu h\|\mathbf{u}\|_2.
$$
Eliminate the symbol $\nu$ from both sides, which yields the error estimate    
\begin{equation}	\label{uh -uI}
      \interleave\mathbf{u}_h - \mathbf{u}_I\interleave \le \left[\frac{C_{\mathcal{R}_1}}{C_*} \left(\|\mathbf{u}\|_2 + \|\kappa^{1/2}\mathbf{u}\|\right) +  \frac{C_{\mathcal{T}_2}}{C_*} \|\mathbf{u}\|_2 \right]  h.
\end{equation}
The estimate \eqref{uh -uI} and the application of the triangle inequality conclude 
\begin{equation}
    \begin{aligned}
        \interleave \mathbf{u} - \mathbf{u}_h\interleave \le &\;\interleave \mathbf{u} - \mathbf{u}_I\interleave + \interleave \mathbf{u}_h - \mathbf{u}_I\interleave\\
        \le &\; \left[\frac{C_{\mathcal{R}_1}}{C_*} \left(\|\mathbf{u}\|_2 + \|\kappa^{1/2}\mathbf{u}\|\right) +  \left(\frac{C_{\mathcal{T}_2}}{C_*} +  C_{I_1} + C_{I_1}\|\kappa^{-1/2}\|_\infty\right) \|\mathbf{u}\|_2 \right]  h \\
       := &\; \widehat{C} \, h.
    \end{aligned}
\end{equation}

Next we should bound the term $\|p - p_h\|$. Let $p_I$ be the piecewise constant projection of $p$ with respect to $\Omega_h$. Using \eqref{Brinkman eq 1} and \eqref{discrete scheme 1}, we arrive at 
\begin{equation}	\label{p error eq}
    \begin{aligned}
        \left(p_I - p_h, \nabla\cdot\mathbf{v}_h\right) =&\; (p_I, \nabla\cdot\mathbf{v}_h) + (\mathbf{f},\mathcal{R}_h\mathbf{v}_h) - \nu a_h(\mathbf{u}_h,\mathbf{v}_h)\\
        =&\;(p_I, \nabla\cdot\mathbf{v}_h) + ( -\nu \mathbf{u} + \nu \kappa^{-1} \mathbf{u} + \nabla p,\mathcal{R}_h\mathbf{v}_h) - \nu a_h(\mathbf{u}_h,\mathbf{v}_h) \\
        =&\;   (p_I, \nabla\cdot\mathbf{v}_h) - (p,\nabla\cdot\mathcal{R}_h\mathbf{v}_h) + \nu a(\mathbf{u}, \mathcal{R}_h\mathbf{v}_h)- \nu a_h(\mathbf{u}_h,\mathbf{v}_h)\\
        :=&\; \mathcal{S}_1 + \mathcal{S}_2.
    \end{aligned}
\end{equation}
For the first term $\mathcal{S}_1$, we observe that $p_I$ represents  a piecewise constant in local element $E$, as indicated by  \eqref{reconstruction 2} and the orthogonality of $L^2$-projection, which leads to
\begin{equation}
    \mathcal{S}_1:=(p_I, \nabla\cdot\mathbf{v}_h) - (p,\nabla\cdot\mathcal{R}_h\mathbf{v}_h) = (p_I - p,\nabla\cdot\mathcal{R}_h\mathbf{v}_h) = 0.
\end{equation}
For the second term $\mathcal{S}_2$, repeating the similar analysis as the velocity error estimate, we deduce that
\begin{equation}	\label{s2 estimate}
    \begin{aligned}
        \mathcal{S}_2: = &\; \nu a(\mathbf{u}, \mathcal{R}_h\mathbf{v}_h)- \nu a_h(\mathbf{u}_h,\mathbf{v}_h)\\
        =&\;\nu a(\mathbf{u}, \mathcal{R}_h\mathbf{v}_h - \mathbf{v}_h) + \nu a(\mathbf{u},\mathbf{v}_h)
        - \nu a_h(\mathbf{u},\mathbf{v}_h) +  \nu a_h(\mathbf{u} -\mathbf{u}_h,\mathbf{v}_h) \\
        =&\;\nu a(\mathbf{u}, \mathcal{R}_h\mathbf{v}_h - \mathbf{v}_h) + \nu a(\mathbf{u}-\mathbf{u}_\pi,\mathbf{v}_h) - \nu a_h(\mathbf{u}-\mathbf{u}_\pi,\mathbf{v}_h)  +  \nu a_h(\mathbf{u} -\mathbf{u}_h,\mathbf{v}_h)\\
        \le &\; \nu\left(C_a \sqrt{C_{\mathcal{R}_1}^2 + C_{\mathcal{R}_2}^2} \interleave\mathbf{u}\interleave + (C_a + C^*) C_\pi \|\mathbf{u}\|_2(1+\|\kappa^{-1/2}\|_\infty)+ \widehat{C}\,\right) h\interleave\mathbf{v}_h\interleave\\
        :=&\; \nu C_{s}h\interleave\mathbf{v}_h\interleave.
    \end{aligned}
\end{equation}
Then collecting the above \eqref{p error eq}-\eqref{s2 estimate}, and by using the discrete inf-sup condition \eqref{discrete inf-sup}, it lead to
\begin{equation}	\label{eq pI- ph}
    \|p_I - p_h\| \le\frac{1}{\widetilde{\gamma}}  \sup_{\mathbf{v}_h\in \mathbf{V}_h/ {\{\mathbf{0}\}}} \frac{d(\mathbf{v}_h, p_I - p_h)}{\interleave \mathbf{v}_h\interleave} \le \nu\frac{C_{s}}{\widetilde{\gamma}} h.
\end{equation}
Finally, using triangle inequality, \eqref{eq pI- ph} and  Lemma \ref{lemma calssical interpolation}, we arrive at
\begin{equation}
     \|p - p_h\| \le \|p- p_I\| + \|p_I - p_h\|\le\left( C_{I_2} \|p\|_1 + \nu\frac{ C_{s}}{\widetilde{\gamma}}\right) h:= \widetilde{C} h.
\end{equation}
The proof is completed.
\end{proof}

\begin{remark}	\label{remark robust}
    The really pressure-robust virtual element is Locking-free for $\nu\rightarrow 0$, and it eliminates the pressure-dependence of the velocity approximation. From the above theorem, we can observe that the constant $\widehat{C}$ is independent of pressure $p$ and even of viscosity $\nu$ in \eqref{priori estimate u}. The robustness will be demonstrated through subsequently numerical examples.
\end{remark}

From the above Theorem \ref{thm prior u1 and p} and the definition of norm $\interleave\cdot\interleave$, we have the estimate of velocity in $[H^1(\Omega)]^2$ semi-norm.
\begin{coro}
     Under the same Assumption of Theorem \ref{thm prior u1 and p}, the priori estimate of velocity  in $[H^1(\Omega)]^2$ semi-norm as follow
     $$
     |\mathbf{u} - \mathbf{u}_h|_1 \le \widehat{C} h,
     $$
     where the positive constant $\widehat{C}$ is the same as in  Theorem \ref{thm prior u1 and p}.
\end{coro}
\section{A posteriori error analysis and mesh adaptivity}
    In the section, we present a residual-type posteriori error estimator, and prove the estimator yields globally upper and locally lower bounds for the discretization error. Finally, we proposed a strategy of mesh refinement and an adaptive algorithm for VEM approximate to Brinkman problem.
    
   Firstly, we deduce the residual  error equation. Let $(\mathbf{u}, p)\in\mathbf{V}\times Q$ and $(\mathbf{u}_h, p_h)\in\mathbf{V}_h\times Q_h$ be the solutions of \eqref{continue weak formulation} and \eqref{discrete scheme}, respectively. we denote the error as
    $$
    e_{\mathbf{u}}:= \mathbf{u} - \mathbf{u}_h,\qquad e_p:= p - p_h.
    $$
    Then we can obtain the  error equation :
    \begin{equation}	\label{post error eq ab}
        \begin{aligned}
            &\nu a(e_{\mathbf{u}}, \mathbf{v}) + b(\mathbf{v},e_p) = (\mathbf{f},\mathbf{v}) - \nu a(\mathbf{u}_h,\mathbf{v}) - b(\mathbf{v},p_h)		\\
            & = (\mathbf{f},\mathbf{v}) - (\mathbf{f},\mathcal{R}_h\mathbf{v}_h ) - \nu a(\mathbf{u}_h,\mathbf{v}) - b(\mathbf{v},p_h) + \nu a_h(\mathbf{u}_h,\mathbf{v}_h) + b(\mathbf{v}_h, p_h) \\
            & = (\mathbf{f},\mathbf{v}_h - \mathcal{R}_h\mathbf{v}_h)  + \nu a_h(\mathbf{u}_h,\mathbf{v}_h) - \nu a(\mathbf{u}_h,\mathbf{v}_h)  - \nu a(\mathbf{u}_h, \mathbf{v} - \mathbf{v}_h) - b(\mathbf{v} - \mathbf{v}_h, p_h) + (\mathbf{f},\mathbf{v} - \mathbf{v}_h),
        \end{aligned}      
    \end{equation}
    where for the last three terms, it follows from the integration by parts that
    \begin{equation}
        \begin{aligned}
                &- \nu a(\mathbf{u}_h, \mathbf{v} - \mathbf{v}_h) - b(\mathbf{v} - \mathbf{v}_h, p_h) + (\mathbf{f},\mathbf{v} - \mathbf{v}_h)\\
                 =& -\nu \left(\nabla\Pi^{\nabla,h}\mathbf{u}_h, \nabla(\mathbf{v}-\mathbf{v}_h)\right) -  \nu \left(\nabla(\mathbf{u}_h - \Pi^{\nabla,h}\mathbf{u}_h), \nabla(\mathbf{v}-\mathbf{v}_h)\right) - \nu\left(\kappa^{-1}\Pi^{0,h}\mathbf{u}_h, \mathbf{v} - \mathbf{v}_h\right) \\
                & -\nu\left(\kappa^{-1}(\mathbf{u}_h - \Pi^{0,h}\mathbf{u}_h), \mathbf{v} - \mathbf{v}_h\right) + \left(p_h,\nabla\cdot(\mathbf{v} - \mathbf{v}_h)\right) + (\mathbf{f},\mathbf{v} - \mathbf{v}_h) \\
                = & -  \nu \left(\nabla(\mathbf{u}_h - \Pi^{\nabla,h}\mathbf{u}_h), \nabla(\mathbf{v}-\mathbf{v}_h)\right)  -\nu\left(\kappa^{-1}(\mathbf{u}_h - \Pi^{0,h}\mathbf{u}_h), \mathbf{v} - \mathbf{v}_h\right) \\
                & + ( - \nu \kappa^{-1}\Pi^{0,h}\mathbf{u}_h + \mathbf{f}, \mathbf{v} - \mathbf{v}_h) - 
                \sum_{E \in \Omega_h}\sum_{e\in\mathcal{F}^0_h\cap\partial E}\left([\nu\nabla\Pi^{\nabla,E}\mathbf{u}_h - p_h \mathbf{I}]\mathbf{n}_e, \mathbf{v} - \mathbf{v}_h\right),
        \end{aligned}
    \end{equation}
where $\Delta \Pi^{\nabla,E}\mathbf{u}_h = \mathbf{0}$ and $\nabla p_h = \mathbf{0}$, due to $\Pi^{\nabla,E}\mathbf{u}_h \in \mathbf{P}_1(E)$ and $p_h$  belonging to the set of piecewise $P_0(E)$ (denoted as $Q_h$). Furthermore, we can derive the final form of error equation \eqref{post error eq ab} as follows:
    \begin{equation}	\label{posteriori error eq}
        \begin{aligned}
            &\nu a(e_{\mathbf{u}}, \mathbf{v}) + b(\mathbf{v},e_p)\\
             =\; &(\mathbf{f},\mathbf{v}_h - \mathcal{R}_h\mathbf{v}_h)  + \nu a_h(\mathbf{u}_h,\mathbf{v}_h) - \nu a(\mathbf{u}_h,\mathbf{v}_h) \\
            &-  \nu \sum_{E \in \Omega_h}\left(\nabla(\mathbf{u}_h - \Pi^{\nabla,E}\mathbf{u}_h), \nabla(\mathbf{v}-\mathbf{v}_h)\right)_E -\nu\sum_{E \in \Omega_h}\left(\kappa^{-1}(\mathbf{u}_h - \Pi^{0,E}\mathbf{u}_h), \mathbf{v} - \mathbf{v}_h\right)_E \\
            & + \sum_{E \in \Omega_h}( - \nu \kappa^{-1}\Pi^{0,E}\mathbf{u}_h + \mathbf{f}, \mathbf{v} - \mathbf{v}_h)_E - \sum_{E \in \Omega_h}\sum_{e\in\mathcal{F}^0_h\cap\partial E}\left([\nu\nabla\Pi^{\nabla,E}\mathbf{u}_h - p_h \mathbf{I}]\,\mathbf{n}_e, \mathbf{v} - \mathbf{v}_h\right)_e\\
            :=& \sum_{i=1}^{5} \mathcal{J}_i.
        \end{aligned}
    \end{equation}
    
Inspired by \cite{Wang2020post, Wang2021post}, we construct the local  estimator $\eta_E$ based on  \eqref{posteriori error eq} as follows
\begin{equation}
    \eta_E^2 : = \eta_{\mathbf{f},E}^2 + \eta_{S,E}^2 + \eta_{r,E}^2,
\end{equation}
where
\begin{subequations}
    \begin{align}
        &\eta_{\mathbf{f},E}^2:= h^2_E\,\|\mathbf{f}\|_E^2,	\\
        &\eta_{S,E}^2 := \nu^2S^\nabla_E((I- \Pi^{\nabla,E})\mathbf{u}_h, (I- \Pi^{\nabla,E})\mathbf{u}_h) + \nu^2S^0_E((I- \Pi^{0,E})\mathbf{u}_h, (I- \Pi^{0,E})\mathbf{u}_h),	\\
        &  \eta_{r,E}^2:= \nu^2 h_E^2\,\| \kappa^{-1}\Pi^{0,E}\mathbf{u}_h\|^2_E + \frac{1}{2} \sum_{e\in\mathcal{F}^0_h\cap\partial E} h_e\left\|[\nu\nabla\Pi^{\nabla,E}\mathbf{u}_h - p_h \mathbf{I}]\,\mathbf{n}_e \right\|_e^2.
    \end{align}
\end{subequations}
The global error estimator $\eta$ is defined as 
$$
\eta^2 = \sum_{E \in \Omega_h} \eta_E^2.
$$
Before analyzing the posteriori error, we present an local interpolation result \cite{Andrea2017post, Veiga2017Stokes}.
\begin{lema}	\label{lemma interpolation H1}
    Under  Assumption \ref{assumption regularity}, for $\mathbf{v}\in [H^1(\Omega)]^2$, there exists a interpolation function $\mathbf{v}_I\in \mathbf{V}_h$, such that for all elements $E\in\Omega_h$, it holds that 
    \begin{equation}	\label{interpolation H1}
        \|\mathbf{v} - \mathbf{v}_I\|_E + h_E |\mathbf{v} - \mathbf{v}_I|_{1,E} \le \widetilde{C}_{I}\,h_E |\mathbf{v}|_{1,\widetilde{E}},
    \end{equation}
  where  $\widetilde{E}$ denotes the union of the polygons in $\Omega_h$ intersecting $E$, and the positive constant $\widetilde{C}_{I}$ depend only on the mesh regularity.
\end{lema}

\subsection{Upper bound}
In this subsection, we will provide a global upper bound for the discretization error  $\nu\interleave \mathbf{u} - \mathbf{u}_h\interleave + \|p-p_h\|$, based on the error equation \eqref{posteriori error eq}.

\begin{thm}
        (Reliability). Let $(\mathbf{u}, p)\in\mathbf{V}\times Q$ and $(\mathbf{u}_h, p_h)\in\mathbf{V}_h\times Q_h$ be the solutions of \eqref{continue weak formulation} and \eqref{discrete scheme}, respectively. Then we can derive the following upper bound
        \begin{equation}
            \nu\interleave \mathbf{u} - \mathbf{u}_h\interleave + \|p-p_h\| \le C_{\eta}\,\eta,
        \end{equation}
        where the positive constant $C_{\eta}:= C_{\eta,u} + C_{\eta,p}$ is independent of $h$.
    \end{thm}
    \begin{proof}
        Let $e_{\mathbf{u},I}\in\mathbf{V}_h$  be the interpolation of $e_\mathbf{u}$ satisfying the  Lemma \eqref{lemma calssical interpolation}. Taking $\mathbf{v} = e_\mathbf{u}, \mathbf{v}_h = e_{\mathbf{u},I}$ in \eqref{posteriori error eq}, we now estimate $\mathcal{J}_i$ term by term. For the first term $\mathcal{J}_1$, by using  \eqref{reconstruction 3} and \eqref{interpolation H1}, it follows that
        \begin{equation}	\label{J1}
            \begin{aligned}
                \mathcal{J}_1 &\le C_{\mathcal{R}_1} h\,\|\mathbf{f}\| | e_{\mathbf{u},I} |_{1} \le  C_{\mathcal{R}_1} h\,\|\mathbf{f}\|  \left( |e_{\mathbf{u},I} -e_{\mathbf{u}}|_{1} + |e_{\mathbf{u}} |_{1} \right) \\
                &\le C_{\mathcal{R}_1} h\,\|\mathbf{f}\|  \left(\widetilde{C}_{I}C_{\widetilde{E}}|e_{\mathbf{u}}|_1 + |e_{\mathbf{u}}|_1\right)\\
                &\le C_{\mathcal{R}_1}\left(\widetilde{C}_{I}C_{\widetilde{E}} + 1\right) h\,\|\mathbf{f}\|  |e_{\mathbf{u}}|_1 \\
                & := C_{\mathcal{J}_1} h\,\|\mathbf{f}\|  \interleave e_{\mathbf{u}}\interleave,
            \end{aligned}
        \end{equation}
where the  $C_{\widetilde{E}}$ denotes a positive constant depending only on the mesh regularity.
       
For the second term $\mathcal{J}_2$, a use of Lemma \ref{lemma calssical polynomial estimate} and the properties \eqref{stabilzation eq} of the stabilizing terms $S^\nabla_E$ and $S^0_E$  yield to
\begin{equation}
    \begin{aligned}
        \mathcal{J}_2 := & \;\nu a_h(\mathbf{u}_h,e_{\mathbf{u},I}) - \nu a(\mathbf{u}_h,e_{\mathbf{u},I}) \\
        = & -\nu \left(\nabla(\mathbf{u}_h - \Pi^{\nabla, h}\mathbf{u}_h),\nabla(e_{\mathbf{u},I} - \Pi^{\nabla, h}e_{\mathbf{u},I})\right) - \nu \left(\kappa^{-1}(I-\Pi^{0,h})\mathbf{u}_h, (I- \Pi^{0,h})e_{\mathbf{u},I}\right)\\
        & + \nu\sum_{E \in \Omega_h} S^{\nabla}_E \left((I- \Pi^{\nabla,E})\mathbf{u}_h, (I- \Pi^{\nabla,E})e_{\mathbf{u},I}\right) + S^0_E\left((I- \Pi^{0,E})\mathbf{u}_h, (I- \Pi^{0,E})e_{\mathbf{u},I}\right) \\
        \le &\;\nu (1 + C^\nabla_2) \| \nabla(\mathbf{u}_h - \Pi^{\nabla,h}\mathbf{u}_h)\|\,\|\nabla(e_{\mathbf{u},I} - \Pi^{\nabla,h}e_{\mathbf{u},I})\| \\
        &\;+ \nu (1 + C^0_2)\|\kappa^{-1/2}(\mathbf{u}_h - \Pi^{0,h}\mathbf{u}_h)\|\,\|\kappa^{-1/2}(e_{\mathbf{u},I} - \Pi^{0,h}e_{\mathbf{u},I})\|\\
        \le &\; \nu \frac{(1+ C_2^{\nabla})C_\pi}{(C_1^{\nabla})^{1/2}} \sum_{E \in \Omega_h} \left(S^\nabla_E((I- \Pi^{\nabla,E})\mathbf{u}_h, (I- \Pi^{\nabla,E})\mathbf{u}_h)\right)^{1/2} |e_{\mathbf{u},I}|_1\\
        &\; + \nu \frac{(1+ C_2^{0})C_\pi}{(C_1^{0})^{1/2}} h\,\sum_{E \in \Omega_h} \left(S^0_E((I- \Pi^{0,E})\mathbf{u}_h, (I- \Pi^{0,E})\mathbf{u}_h)\right)^{1/2}\|\kappa^{-1/2}\|_{\infty}|e_{\mathbf{u},I}|_1 \\
        \le &\; \nu  C_{\mathcal{J}_2}\interleave e_{\mathbf{u}}\interleave \sum_{E \in \Omega_h} \left(S^\nabla_E((I- \Pi^{\nabla,E})\mathbf{u}_h, (I- \Pi^{\nabla,E})\mathbf{u}_h)\right)^{1/2} + \left(S^0_E((I- \Pi^{0,E})\mathbf{u}_h, (I- \Pi^{0,E})\mathbf{u}_h)\right)^{1/2},
    \end{aligned}
\end{equation}
where the positive constant $C_{\mathcal{J}_2}:= \left(\widetilde{C}_{I}C_{\widetilde{E}} + 1\right)\max\left\{\frac{(1+ C_2^{\nabla})C_\pi}{(C_1^{\nabla})^{1/2}}, \frac{(1+ C_2^{0})C_\pi}{(C_1^{0})^{1/2}}\|\kappa^{-1/2}\|_{\infty}\right\} $ is a positive constant independent of $h$,  and the  same trick as used in \eqref{J1} is employed in last step.

Similar to the proof of $\mathcal{J}_2$, one can derive the bound of $\mathcal{J}_3$ as follows
\begin{equation}
    \mathcal{J}_3\le \nu C_{\mathcal{J}_3} \interleave e_{\mathbf{u}}\interleave \sum_{E \in \Omega_h} \left(S^\nabla_E((I- \Pi^{\nabla,E})\mathbf{u}_h, (I- \Pi^{\nabla,E})\mathbf{u}_h)\right)^{1/2} + \left(S^0_E((I- \Pi^{0,E})\mathbf{u}_h, (I- \Pi^{0,E})\mathbf{u}_h)\right)^{1/2},
\end{equation}
 where the positive constant  $C_{\mathcal{J}_3}: = \widetilde{C}_{I}C_{\widetilde{E}} \max\left\{\frac{1}{(C^\nabla_1)^{1/2}}, \frac{\|\kappa^{-1/2}\|_{\infty}}{(C^0_1)^{1/2}}\right\}$ is a positive constant independent of $h$.

The estimate of $\mathcal{J}_4$ also can be easily carried out
\begin{equation}
    \begin{aligned}
        \mathcal{J}_4:=&\;\sum_{E \in \Omega_h}( - \nu \kappa^{-1}\Pi^{0,E}\mathbf{u}_h + \mathbf{f}, e_{\mathbf{u}} - e_{\mathbf{u},I})_E	\\
        \le &\; \| - \nu \kappa^{-1}\Pi^{0,h}\mathbf{u}_h + \mathbf{f}\|\;\|e_{\mathbf{u}} - e_{\mathbf{u},I}\|	\\
        \le&\; \widetilde{C}_{I}C_{\widetilde{E}} \,h\,\| - \nu \kappa^{-1}\Pi^{0,h}\mathbf{u}_h + \mathbf{f}\|\interleave e_\mathbf{u}\interleave\\
        := &\; C_{\mathcal{J}_4} h\,\left(\nu\| \kappa^{-1}\Pi^{0,h}\mathbf{u}_h \| + \| \mathbf{f}\|\right)\interleave e_\mathbf{u}\interleave.
    \end{aligned}
\end{equation} 
Before estimating $\mathcal{J}_5$, we use Young's inequality to give a variant of the trace inequality: $\exists\,C_{\text{tra}}$ that is independent of $h_E$ such that
$$
\|w\|_e \le \frac{\sqrt{2}C_{\text{tra}} }{2} \left(h_E^{-1} \|w\|^{2}_E + h_E |w|_{1,E}^2 \right)^{1/2}\qquad \text{for any edge } e \in \partial E.
$$
Finally, by using Cauchy-Schwartz inequality, trace inequality and Lemma \ref{lemma interpolation H1}, we arrive at the bound  of the last term $\mathcal{J}_5$ as follows
$$	
    \begin{aligned}
        \mathcal{J}_5\le&\;\sum_{E \in \Omega_h}\sum_{e\in\mathcal{F}^0_h\cap\partial E} \left\|[\nu\nabla\Pi^{\nabla,E}\mathbf{u}_h - p_h \mathbf{I}]\,\mathbf{n}_e \right\|_e\;\|e_{\mathbf{u}} - e_{\mathbf{u},I}\|_e\\
        \le &\; \left(\sum_{E \in \Omega_h}\sum_{e\in\mathcal{F}^0_h\cap\partial E} \left\|[\nu\nabla\Pi^{\nabla,E}\mathbf{u}_h - p_h \mathbf{I}]\,\mathbf{n}_e \right\|_e^2 \right)^{1/2} \; \left(\sum_{E \in \Omega_h}\sum_{e\in\mathcal{F}^0_h\cap\partial E}\|e_{\mathbf{u}} - e_{\mathbf{u},I}\|_e^2\right)^{1/2}\\        
        \le&\;\frac{\sqrt{2}C_{\text{tra}} }{2}  \left(\sum_{E \in \Omega_h}\sum_{e\in\mathcal{F}^0_h\cap\partial E} \left\|[\nu\nabla\Pi^{\nabla,E}\mathbf{u}_h - p_h \mathbf{I}]\,\mathbf{n}_e \right\|_e^2 \right)^{1/2} \left(\sum_{E \in \Omega_h} h_E^{-1} \|e_{\mathbf{u}} - e_{\mathbf{u},I}\|^{2}_E + h_E |e_{\mathbf{u}} - e_{\mathbf{u},I}|_{1,E}^2 \right)^{1/2}\\
        \le&\;C_{\text{tra}}\widetilde{C}_{I} \left(\sum_{E \in \Omega_h}\sum_{e\in\mathcal{F}^0_h\cap\partial E} \left\|[\nu\nabla\Pi^{\nabla,E}\mathbf{u}_h - p_h \mathbf{I}]\,\mathbf{n}_e \right\|_e^2 \right)^{1/2} \left(\sum_{E \in \Omega_h} h_E\, |e_\mathbf{u}|_{1,\widetilde{E}}^2  \right)^{1/2} .  
    \end{aligned}
$$
Recalling the Assumption \ref{assumption regularity}, we arrive at
\begin{equation}	\label{J5}
    \begin{aligned}
        \mathcal{J}_5\le&\;\frac{\sqrt{2}C_{\text{tra}}\widetilde{C}_{I}C_{\widetilde{E}}}{\sqrt{\rho}}\left(\frac{1}{2}\sum_{E \in \Omega_h}\sum_{e\in\mathcal{F}^0_h\cap\partial E} h_e\left\|[\nu\nabla\Pi^{\nabla,E}\mathbf{u}_h - p_h \mathbf{I}]\,\mathbf{n}_e \right\|_e^2 \right)^{1/2} \interleave e_\mathbf{u}\interleave\\
        := &\;C_{\mathcal{J}_5}\left(\frac{1}{2}\sum_{E \in \Omega_h}\sum_{e\in\mathcal{F}^0_h\cap\partial E} h_e\left\|[\nu\nabla\Pi^{\nabla,E}\mathbf{u}_h - p_h \mathbf{I}]\,\mathbf{n}_e \right\|_e^2 \right)^{1/2} \interleave e_\mathbf{u}\interleave.
    \end{aligned}
\end{equation}

Since $\nabla\cdot \mathbf{u}_h = 0$ and \eqref{continue form b}, then $b(e_\mathbf{u}, e_p) = 0$. Thus, combining  \eqref{J1}-\eqref{J5} to \eqref{posteriori error eq}, we deduce that
\begin{equation}	\label{the estimate of velocity in proof}
    \nu\interleave\mathbf{u} - \mathbf{u}_h\interleave\le C_{\eta,u} \,\eta,
\end{equation} 
where the positive constant $C_{\eta,u}:= \max\left\{C_{\mathcal{J}_1}+C_{\mathcal{J}_4}, 2\left( C_{\mathcal{J}_2} + C_{\mathcal{J}_3}\right), 2\left(C_{\mathcal{J}_4} + C_{\mathcal{J}_5}\right)\right\}$.

To end the proof, we remain to estimate the error for the pressure. It follows from the \eqref{continue weak formulation 1} and \eqref{discrete scheme 1},  the orthogonality of $\Pi^{\nabla,E}$ and $\Pi^{0,E}$  that
\begin{equation}	\label{pressure error eq 1}
    \begin{aligned}
        (\nabla\cdot\mathbf{v}, p- p_h)=&\; b(\mathbf{v},p) - b(\mathbf{v},p_h)\\
        =&\;(\mathbf{f},\mathbf{v}) - \nu a(\mathbf{u},\mathbf{v}) - \left(\nabla\cdot(\mathbf{v} - \mathbf{v}_I), p_h\right) - (\nabla\cdot\mathbf{v}_I, p_h)\\
        =&\;- \nu a(\mathbf{u},\mathbf{v}) + \nu a_h(\mathbf{u}_h, \mathbf{v}_I) +(\mathbf{f},\mathbf{v})- (\mathbf{f},\mathcal{R}_h \mathbf{v}_I)  - \left(\nabla\cdot(\mathbf{v} - \mathbf{v}_I), p_h\right)  \\
        =&\; -\nu\left(\nabla\mathbf{u} - \nabla\Pi^{\nabla,h}\mathbf{u}_h, \nabla\mathbf{v}\right) - \nu\left(\nabla\Pi^{\nabla,h}\mathbf{u}_h, \nabla\mathbf{v} - \nabla\mathbf{v}_I + \nabla\mathbf{v}_I - \nabla\Pi^{\nabla,h}\mathbf{v}_I \right)\\
        &\; -\nu\left(\kappa^{-1}(\mathbf{u} - \Pi^{0,h}\mathbf{u}_h), \mathbf{v}\right) -\nu\left(\kappa^{-1}\Pi^{0,h}\mathbf{u}_h, \mathbf{v} - \mathbf{v}_I + \mathbf{v}_I - \Pi^{0,h}\mathbf{v}_I\right)\\
        &\; +\nu\sum_{E \in \Omega_h}  S^{\nabla}_E \left((I- \Pi^{\nabla,E})\mathbf{u}_h, (I- \Pi^{\nabla,E})\mathbf{v}_I\right) + S^0_E\left((I- \Pi^{0,E})\mathbf{u}_h, (I- \Pi^{0,E})\mathbf{v}_I\right)\\
        &\;- \left(\nabla\cdot(\mathbf{v} - \mathbf{v}_I), p_h\right) + (\mathbf{f},\mathbf{v}) - (\mathbf{f},\mathcal{R}_h \mathbf{v}_I).
    \end{aligned}
\end{equation}
Using the integration by parts, we derive the following error equation of pressure
\begin{equation}	\label{pressure error eq final}
    \begin{aligned}
         &(\nabla\cdot\mathbf{v}, p- p_h)\\
        =&\;   -\nu\left(\nabla(\mathbf{u} - \mathbf{u}_h + \mathbf{u}_h - \Pi^{\nabla,h}\mathbf{u}_h), \nabla\mathbf{v}\right)  -\nu\left(\kappa^{-1}(\mathbf{u} - \mathbf{u}_h + \mathbf{u}_h - \Pi^{0,h}\mathbf{u}_h), \mathbf{v}\right) -\nu\left(\kappa^{-1}\Pi^{0,h}\mathbf{u}_h, \mathbf{v} - \mathbf{v}_I\right)\\
        &\;+\nu\sum_{E \in \Omega_h}  S^{\nabla}_E \left((I- \Pi^{\nabla,E})\mathbf{u}_h, (I- \Pi^{\nabla,E})\mathbf{v}_I\right) + S^0_E\left((I- \Pi^{0,E})\mathbf{u}_h, (I- \Pi^{0,E})\mathbf{v}_I\right)\\
        &\;+ (\mathbf{f},\mathbf{v}-\mathbf{v}_I) - (\mathbf{f},\mathcal{R}_h \mathbf{v}_I - \mathbf{v}_I) -  \sum_{E \in \Omega_h}\sum_{e\in\mathcal{F}^0_h\cap\partial E} \left([\nu \nabla\Pi^{\nabla,E}\mathbf{u}_h - p_h \mathbf{I}]\mathbf{n}_e, \mathbf{v} - \mathbf{v}_I\right).
    \end{aligned}
\end{equation}

Then similar to the proofs of \eqref{J1}-\eqref{J5}, and combining the estimate \eqref{the estimate of velocity in proof} of velocity and the discrete inf-sup condition \eqref{discrete inf-sup}, one can easily derive the desired result for the pressure from error equation \eqref{pressure error eq final}.
$$
\| p - p_h\| \le C_{\eta,p} \,\eta,
$$
where the positive constant $C_{\eta,p}$ is independent of $h$. This concludes the proof.    
\end{proof}
 
 \subsection{Local Lower Bound} Before proving the local lower bound, we need to introduce a important tool called bubble function. Under the mesh regularity assumption, each element $E$ admits a sub-triangulation by connecting each vertex of $E$ to the point $\mathbf{x}_E$ with respect to which $E$ is star-shaped. Then, a bubble function $\bm{\psi}_E\in [H^1_0(E)]^2$ can be constructed piecewise as the sum of the cubic bubble functions on each triangle of the sub-triangulation of $E$. An edge bubble function $\bm{\psi}_e$ for $e\subset\partial E$ is a piecewise quadratic function with a value of $\mathbf{1}$ at the middle point of $e$ and $\mathbf{0}$ on the triangles that do not include $e$ as their boundary. The bubble functions have the following properties \cite{Wang2020post, Wang2021post}
 \begin{lema}
     On each element $E\in\Omega_h$, the bubble function $\bm{\psi}_E$ satisfies 
     \begin{subequations}	\label{bubble properties}
         \begin{align}
            & C_{B_1}\|\mathbf{p}_1\|^2_E \le \left(\bm{\psi}_E\mathbf{p}_1, \mathbf{p}_1\right)_E \le C_{B_2} \|\mathbf{p}_1\|^2_E \qquad \forall\,\mathbf{p}_1\in\mathbf{P}_1(E),\\
            & C_{B_3}\|\mathbf{p}_1\|^2_E \le \|\bm{\psi}_E\mathbf{p}_1\|_E + h_E\|\nabla(\bm{\psi}_E\mathbf{p}_1)\|_E \le C_{B_4} \|\mathbf{p}_1\|^2_E \qquad \forall\,\mathbf{p}_1\in\mathbf{P}_1(E),
         \end{align} 
     \end{subequations}
where the positive constants $C_{B_1},C_{B_2},C_{B_3}$ and $C_{B_4}$ are independent of $h_E$. Moreover, For any $e\subset \partial E$ the edge bubble function $\bm{\psi}_e$ satisfies    
\begin{subequations}	\label{edge bubble properties}
    \begin{align}
        & C_{b_1}\|\mathbf{p}_1\|^2_e \le \left(\bm{\psi}_e\mathbf{p}_1, \mathbf{p}_1\right)_e \le C_{b_2} \|\mathbf{p}_1\|^2_e \qquad \forall\,\mathbf{p}_1\in\mathbf{P}_1(e),\\
        & h^{-1/2}_E\|\bm{\psi}_e\mathbf{p}_1\|_E + h^{1/2}_E\|\nabla(\bm{\psi}_e\mathbf{p}_1)\|_E \le C_{b_3} \|\mathbf{p}_1\|_e \qquad \forall\,\mathbf{p}_1\in\mathbf{P}_1(e),
    \end{align} 
\end{subequations}
where the positive constants $C_{b_1},C_{b_2},C_{b_3}$ and $C_{b_4}$ are independent of $h_E$ and $h_e$, and $\bm{\psi}_e\mathbf{p}_1$ is extended from $e$ to $E$ by the technique introduced in Remark 3.1 of \cite{Mora2017post}.
\end{lema}

\begin{thm} 
    (Efficiency). Let $(\mathbf{u}, p)\in\mathbf{V}\times Q$ and $(\mathbf{u}_h, p_h)\in\mathbf{V}_h\times Q_h$ be the solutions of \eqref{continue weak formulation} and \eqref{discrete scheme}, respectively. Then, the local estimator $\eta_{r,E}$ satisfy
    \begin{equation}
        \eta_{r,E}^2 \le C_{r} \sum_{E \in \mathcal{E}_e} \left(\nu^2 \interleave \mathbf{u} - \mathbf{u}_h\interleave^2_E  + \|p - p_h\|^2_E + \eta_{S,E}^2 + \eta_{\mathbf{f},E}^2 + \nu^2 h_E^2 \|\kappa^{-1}\Pi^{0,E}\mathbf{u}_h\|^2_E\right),
    \end{equation}
\end{thm}
where the positive constant $C_r:= C_e\max\{1,C_\theta\}$ is independent of $h_E$.
\begin{proof}
    To be simple, we use $\theta_E$ to denote the local residual $\kappa^{-1}\Pi^{0,E}\mathbf{u}_h$. We readily observe that  $\theta_E$ is a vector linear polynomial on $E$. By using bubble function $\bm{\psi}_E$ on $E$, we set $\mathbf{v} = \bm{\psi}_E\theta_E$ and $\mathbf{v}_h = 0$ in \eqref{posteriori error eq}, and note that $\bm{\psi}_E$ vanishes on boundary $\partial E$, we arrive at
    \begin{equation}
        \begin{aligned}
            &\nu (\nabla e_{\mathbf{u}}, \nabla\bm{\psi}_E \theta_E)_E + \nu(\kappa^{-1}e_{\mathbf{u}},\bm{\psi}_E \theta_E)_E + \left(\nabla\cdot(\bm{\psi}_E \theta_E), e_p\right)_E \\
            =\; &-\nu\left(\nabla(\mathbf{u}_h - \Pi^{\nabla,E}\mathbf{u}_h),\nabla\bm{\psi}_E \theta_E\right)_E - \nu\left(\kappa^{-1}(\mathbf{u}_h - \Pi^{0,E}\mathbf{u}_h), \bm{\psi}_E \theta_E\right)_E + \left(-\nu \theta_E + \mathbf{f}, \bm{\psi}_E\theta_E\right)_E.
        \end{aligned}
    \end{equation}
Utilizing the properties the properties \eqref{bubble properties} of the bubble function $\bm{\psi}_E$, we obtain
$$
   \begin{aligned}
       C_{B_1}\nu\|\theta_E\|^2_E \le&\; \nu(\theta_E, \bm{\psi}_E\theta_E) \\
       =&\; - \nu (\nabla e_{\mathbf{u}}, \bm{\psi}_E \theta_E)_E - \nu (\kappa^{-1}e_{\mathbf{u}}, \bm{\psi}_E \theta_E)_E - \left(\nabla\cdot(\bm{\psi}_E \theta_E), e_p\right)_E + (\mathbf{f}, \bm{\psi}_E \theta_E)	\\
       &\; -\nu\left(\nabla(\mathbf{u}_h - \Pi^{\nabla,E}\mathbf{u}_h),\nabla\bm{\psi}_E \theta_E\right)_E - \nu\left(\kappa^{-1}(\mathbf{u}_h - \Pi^{0,E}\mathbf{u}_h), \bm{\psi}_E \theta_E\right)_E\\
       \le&\;\nu |e_\mathbf{u}|_{1,E} \,|\bm{\psi}_E \theta_E|_{1,E}+ \nu\|\kappa^{-1/2}\|_{\infty}\|\kappa^{-1/2}e_\mathbf{u}\|_E \|\bm{\psi}_E \theta_E\|_E + \|p-p_h\|_E\, |\bm{\psi}_E \theta_E|_{1,E} \\
      &  + \frac{\nu}{(C^\nabla_1)^{1/2}} S^\nabla_E((I-\Pi^{\nabla,E})\mathbf{u}_h, (I-\Pi^{\nabla,E})\mathbf{u}_h)^{1/2} |\bm{\psi}_E \theta_E|_{1,E} + \|\mathbf{f}\|_E\,\|\bm{\psi}_E \theta_E\|_E \\
      & + \frac{\nu\|\kappa^{-1/2}\|_\infty}{(C^0_1)^{1/2}} S^0_E((I-\Pi^{0,E})\mathbf{u}_h, (I-\Pi^{0,E})\mathbf{u}_h)^{1/2}\|\bm{\psi}_E \theta_E\|_E .
   \end{aligned}
$$
Furthermore,
$$
\nu\|\theta_E\|^2_E \le\; C_{\theta_1}\left(\nu|e_\mathbf{u}|_{1,E} + \|e_p\|_E + \eta_{S,E}\right) h^{-1}_E \|\theta_E\|_E + C_{\theta_2} \left(\nu\|\kappa^{-1/2}e_\mathbf{u}\|_E + \|f\|_E + \eta_{S,E} \right)\|\theta_E\|_E,
$$
where the positive constants $C_{\theta_1}:= \frac{C_{B_4}}{C_{B_1}}\cdot\max\left\{2, \frac{1}{(C^\nabla_1)^{1/2}} \right\}$ and $C_{\theta_2}:=\frac{C_{B_4}}{C_{B_1}}\cdot\max\left\{1, \frac{\|\kappa^{-1/2}\|_\infty}{(C^0_1)^{1/2}}, \|\kappa^{-1/2}\|_\infty \right\}$ are both independent of $h_E$. Then, eliminating  $\|\theta_E\|_E$ and multiplying $h_E$ on both sides give
\begin{equation}	\label{nu h theta_E}
    \nu^2 h^2_E\|\theta_E\|^2_E \le C_\theta\,\left( \nu^2\interleave e_\mathbf{u}\interleave_E^2 + \|e_p\|^2_E + \eta_{S,E}^2 +  \eta_{\mathbf{f},E}^2\right),
\end{equation}   
where the positive constant $C_\theta:= 4 (C_{\theta_1} + C_{\theta_2})^2$, and  the boundedness of the mesh size is used.   
   
On the other hand, we define $\theta_e$ as $[\nu\Pi^{\nabla,E}\mathbf{u}_h - p_h \mathbf{I}]\mathbf{n}_e$, and  subsequently extend $\theta_e$ into $\mathcal{E}_e$ via a constant vector prolongation in the direction normal to $e$ (see \cite{Mora2017post}, Remark 3.1). Let $\bm{\psi}_e$ represent an edge bubble function on $e$. By setting $\mathbf{v} = \bm{\psi}_e\theta_e, \mathbf{v}_h = \mathbf{0}$ in \eqref{posteriori error eq}, we derive the following result:
$$
\begin{aligned}
    &\nu \sum_{E \in \mathcal{E}_e}\left(\nabla e_{\mathbf{u}}, \nabla \bm{\psi}_e\theta_e\right)_E + \nu \sum_{E \in \mathcal{E}_e}\left(\kappa^{-1}e_{\mathbf{u}}, \bm{\psi}_e\theta_e \right) +  \sum_{E \in \mathcal{E}_e}(\nabla\cdot(\bm{\psi}_e \theta_e), e_p)_E\\
     =&\; -\nu\sum_{E \in \mathcal{E}_e} \left(\nabla(\mathbf{u}_h - \Pi^{\nabla,E}\mathbf{u}_h), \nabla \bm{\psi}_e \theta_e\right)_E - \nu\sum_{E \in \mathcal{E}_e} \left(\kappa^{-1}(\mathbf{u}_h - \Pi^{0,E}\mathbf{u}_h), \bm{\psi}_e\theta_e\right)_E\\
     &\; + \sum_{E \in \mathcal{E}_e} \left(-\nu\kappa^{-1}\Pi^{0,E}\mathbf{u}_h + \mathbf{f}, \bm{\psi}_e\theta_e\right)_E - \left(\theta_e, \bm{\psi}_e\theta_e\right)_e.
\end{aligned}
$$
Using the properties \eqref{edge bubble properties},   we can deduce that 
$$
\begin{aligned}
   C_{b_1} \|\theta_e\|^2_e \le \left(\theta_e, \bm{\psi}_e\theta_e\right)_e \le&\; C_{e_1} \sum_{E \in \mathcal{E}_e} \left(\nu|e_{\mathbf{u}}|_{1,E} + \|e_p\|_E + \eta_{S,E}\right)|\bm{\psi}_e \theta_e|_{1,E} \\
   &\;+ C_{e_2}\sum_{E \in \mathcal{E}_e} \left(\nu \|\kappa^{-1/2}e_\mathbf{u}\|_E + \nu\|\kappa^{-1}\Pi^{0,E}\mathbf{u}_h\| + \|\mathbf{f}\|_E + \eta_{S,E}\right)\|\bm{\psi}_e \theta_e\|_E\\
   \le&\;C_{e_1} C_{b_3}\sum_{E \in \mathcal{E}_e} h^{-1/2}_E \left(\nu|e_{\mathbf{u}}|_{1,E} + \|e_p\|_E + \eta_{S,E}\right)\|\theta_e\|_{e} \\
   &\;+ C_{e_2}C_{b_3}\sum_{E \in \mathcal{E}_e} h^{1/2}_E \left(\nu\|\kappa^{-1/2}e_\mathbf{u}\|_E + \nu\|\theta_E\| + \|\mathbf{f}\|_E + \eta_{S,E}\right)\|\theta_e\|_e,
\end{aligned}
$$
where $C_{e_1}:= \max\left\{2, \frac{1}{(C^\nabla_1)^{1/2}} \right\}$   and $C_{e_2}:= \max\left\{1, \frac{\|\kappa^{-1/2}\|}{(C^0_1)^{1/2}}, \|\kappa^{-1/2}\|\right\}$. Above, multiplying $h_e$ on both sides yields
\begin{equation}
    \begin{aligned}
            h_e \|\theta_e\|^2_e \le &\;\frac{C_{e_1} C_{b_3}}{C_{b_1}}\sum_{E \in \mathcal{E}_e}  \left(\nu|e_{\mathbf{u}}|_{1,E} + \|e_p\|_E + \eta_{S,E}\right)h_e^{1/2}\|\theta_e\|_{e}\\
            &\; + \frac{C_{e_2}C_{b_3}}{C_{b_1}}\sum_{E \in \mathcal{E}_e}  h_E \left(\nu\|\kappa^{-1/2}e_\mathbf{u}\|_E + \nu\|\theta_E\| + \|\mathbf{f}\|_E + \eta_{S,E}\right)h_{e}^{1/2}\|\theta_e\|_e.
    \end{aligned}
\end{equation}
By cancelling $h^{1/2}_e  \|\theta_e\|$, it follows that there exits a  positive constant $C_e$,  which is related to $\frac{C_{e_1} C_{b_3}}{C_{b_1}}$ and $\frac{C_{e_2}C_{b_3}}{C_{b_1}}$ and independent of $h_E$ and $h_e$, such that
\begin{equation}	\label{he theta_e}
    \frac{1}{2}h_e\|\theta_e\|^2_e \le C_e \sum_{E \in \mathcal{E}_e}\left(\nu^2\interleave e_\mathbf{u}\interleave^2_E + \|e_p\|^2_E + \eta_{S,E}^2 + \eta_{\mathbf{f},E}^2 + \nu^2 h_E^2 \|\theta_E\|^2_E \right).
\end{equation}
By summing over all edges $e\subset \partial E$ for  \eqref{he theta_e}, together with \eqref{nu h theta_E}, we can deduce that
 \begin{equation}
     \eta_{r,E}^2 \le C_r \sum_{E \in \mathcal{E}_e} \left(\nu^2 \interleave e_{\mathbf{u}}\interleave^2_E  + \|e_p\|^2_E + \eta_{S,E}^2 + \eta_{\mathbf{f},E}^2 + \nu^2 h_E^2 \|\theta_E\|^2_E\right), 
 \end{equation}
where $C_r:= C_e\max\{1,C_\theta\}$ is independent of $h_E$. The proof is completed.    
\end{proof}

 \subsection{An adaptive mesh refinement}
 We perform refinement on the local polygon as shown in Fig. \ref{refine mesh}. By connecting the centroid of the heptagon with the midpoint of each edge, we obtain a set of sub-quadrilaterals. It should be noted that when there is a hanging point on an edge, to prevent the formation of degenerate triangles during the refinement process, we directly connect the centroid to the hanging point, rather than connecting the midpoint of the hanging point and the endpoint. Please refer to the mVEM package \cite{mvem} for more details. Then we introduce a adaptive algorithm as follows:
 \begin{figure}[ht]
     \centering  
     \begin{tikzpicture}
         \coordinate (V1) at (0, 0);
         \coordinate (V2) at (2, 0);
         \coordinate (V3) at (3.5, 1.5);       
         \coordinate (V4) at (3.5, 2.9);
         \coordinate (V5) at (1.0, 3.8);
         \coordinate (V6) at (0, 2.9);
         \coordinate (V7) at (-1, 2);
         
         \draw (V1) -- (V2) -- (V3) -- (V4) -- (V5) -- (V6) -- (V7) -- cycle;     
         \foreach \i in {1,2,...,7} {
             \node[circle, fill, inner sep=1.5pt] at (V\i) {};
         }
         \pgfmathsetmacro{\cx}{(0 + 2 + 3.5 + 3.5 + 1.0 + 0 + (-1))/7}
         \pgfmathsetmacro{\cy}{(0 + 0 + 1.5 + 2.9 + 3.8 + 2.9 + 2)/7}
         \coordinate (Centroid) at (\cx, \cy);
         \coordinate (M1) at ($(V1)!0.5!(V2)$);
         \coordinate (M2) at ($(V2)!0.5!(V3)$);
         \coordinate (M3) at ($(V3)!0.5!(V4)$);
         \coordinate (M4) at ($(V4)!0.5!(V5)$);
         \coordinate (M5) at ($(V5)!0.5!(V6)$);
         \coordinate (M6) at ($(V6)!0.5!(V7)$);
         \coordinate (M7) at ($(V7)!0.5!(V1)$);
         
         \draw[dashed] (M1) -- (Centroid);
         \draw[dashed] (M2) -- (Centroid);
         \draw[dashed] (M3) -- (Centroid);
         \draw[dashed] (M4) -- (Centroid);
         \draw[dashed] (V6) -- (Centroid);
         \draw[dashed] (M7) -- (Centroid);
         \node[circle, fill, red, inner sep=1.5pt] at (Centroid) {};
     \end{tikzpicture}
     \caption{Illustration  of refining a local heptagon.}
     \label{refine mesh}
 \end{figure}
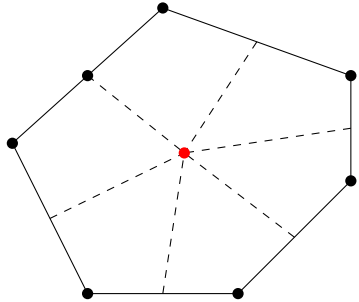
 
 Given an initial coarse mesh $\Omega_h$, we compute the numerical solution $\left(\mathbf{u}_h , p_h\right)$ of \eqref{discrete scheme}, and the local estimator $\eta_E$ on each element of the mesh. Then, we choose a set  $\mathcal{T}_h \subset \Omega_h$ with minimum number such that
 \begin{equation}
     \sum_{E\in \mathcal{T}_h} \eta_E^2\ge \;\delta\sum_{E\in \Omega_h} \eta_E^2 ,
 \end{equation}
 where $\delta \in (0,1)$ is a  marking parameter. We can briefly summarize this procedure with the well-known four steps: 
 $$
 \text{Solve} \rightarrow \text{Estimate}\rightarrow \text{Mark} \rightarrow \text{Refine}.
 $$
 The loop is stopped by a user-specified iteration number $\mathcal{L}$, or until the number $\mathcal{N}$ of nodes in the mesh satisfies $\mathcal{N}\ge \textbf{tol}$, where \textbf{tol} is the user-specified tolerance of nodal number.

\section{Numerical results}
    In this section we present some examples illustrate the method's practical performance. As we emphasized in \cite{xiong2024}, we cannot pointwise access to the numerical solution of velocity $\mathbf{u}_h\in
    \mathbf{V}_h$ within the element. Consequently, we use the projections $\Pi^{\nabla,h} \mathbf{u}_h$ and $\Pi^{0,h} \mathbf{u}_h$ as a substitute for $\mathbf{u}_h$ to calculate the error. Considering the computable error quantity
    $$
     \left(\sum_{E \in \Omega_h}\left\| \nabla\mathbf{u}- \nabla{\Pi}^{\nabla, E} \mathbf{u}_h\right\|_{0, E}^2 + \left\| \kappa^{-1/2}\left(\mathbf{u}- {\Pi}^{0, E} \mathbf{u}_h\right)\right\|_{0, E}^2\right)^{1 / 2}
    $$
    as approximation of $\interleave\mathbf{u} - \mathbf{u}_h\interleave$. And for the error of pressures we directly compute
    $$
   \|p-p_h\|:=\left(\sum_{E \in \Omega_h}\left\|p - p_h\right\|_{0, E}^2\right)^{1 / 2}.
    $$
 Moreover, three types of polygonal meshes will be used in the following Numerical examples, that is  non-convex mesh, square mesh and Voronoi mesh generated by the PolyMesher \cite{polymesher2012}  (see Fig. \ref{somes meshes} for illustration).
\begin{figure}[htbp] 
    \centering  
    \vspace{-0.1cm} 
    \subfigtopskip=2pt 
    \subfigbottomskip=5pt 
    \subfigcapskip=0pt 
    \subfigure[Non-convex mesh]{
        \label{level.sub.non-convex}
        \includegraphics[width=0.31\linewidth]{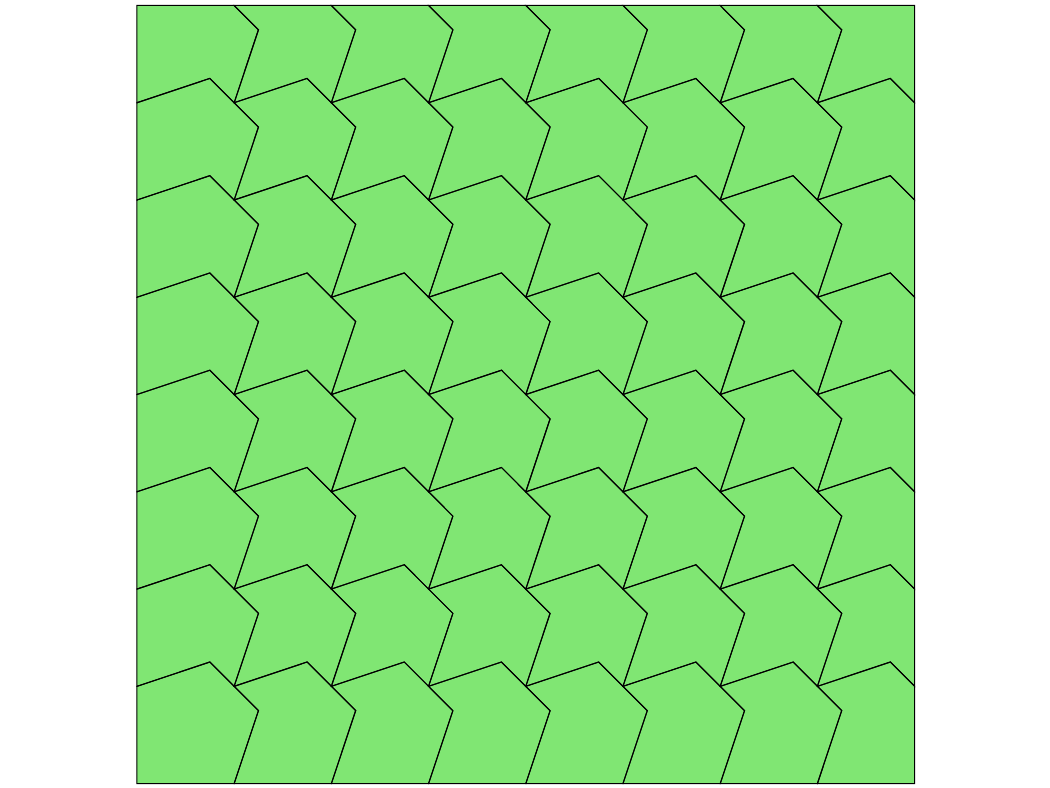}}         
    \subfigure[Square mesh]{
        \label{level.sub.rec}
        \includegraphics[width=0.31\linewidth]{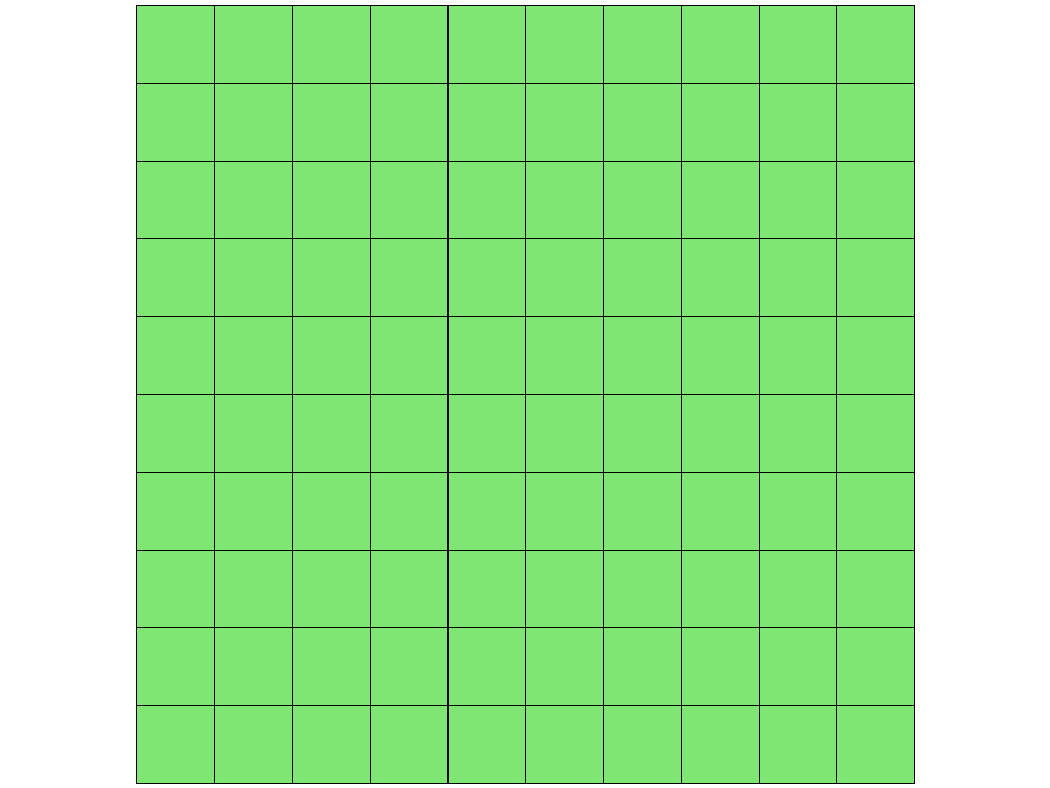}}          
    \subfigure[Voronoi mesh]{
        \label{level.sub.poly}
        \includegraphics[width=0.31\linewidth]{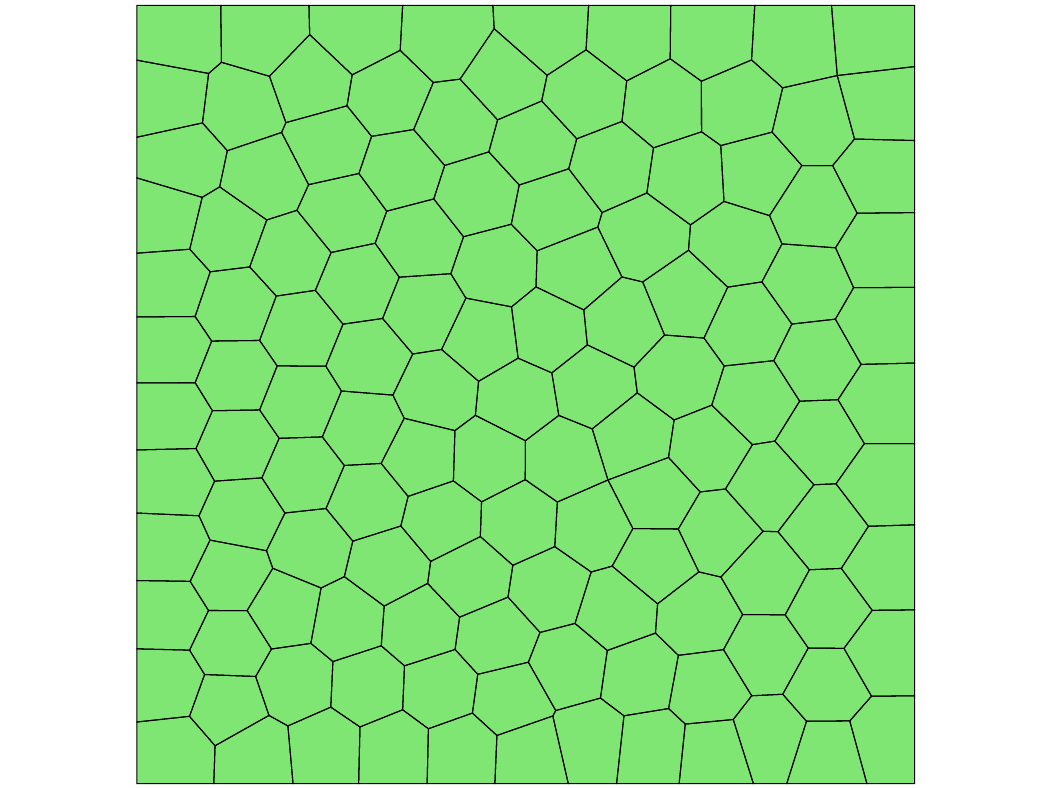}}
    
    \caption{Illustration of three types of meshes.}
    \label{somes meshes}
\end{figure}

\subsection{Numerical tests for a priori error estimate}
In this subsection, three numerical experiments are investigated  to validate our theorem \ref{thm prior u1 and p}  and  numerically simulate the flow of fluids in practical problems. 
\begin{exam}\rm{Let the domain be square $\Omega = (0, 1)\times (0, 1)$, which is partitioned into non-convex mesh Fig. \ref{level.sub.non-convex} . To study the robustness of our virtual element methods to the parameter $\nu$, we fixed $\kappa = 1$ and select different $\nu = 1,10^{-4}, 10^{-8}, 10^{-12}$.  The exact solution for \eqref{Brinkman equations} is chosen as follows : 
    \begin{equation}
        \mathbf{u} = \left(\begin{aligned}
            \sin(2\pi x)\cos(2\pi y)\\
            -\cos(2\pi x)\sin(2\pi y)
        \end{aligned}\right),\quad p = x^2 y^2 - 1/9,
    \end{equation}
where the load function $\mathbf{f}$ is suitably chosen, and it is easy to check that $\nabla\cdot\mathbf{u} = 0$ and $(p,1) = 0$.
    
We show the numerical results on  Table \ref{table1 error for E1}. Compared to the standard VEM scheme, our pressure-robust VEM scheme is Locking-free for viscosity $\nu \rightarrow 0$. And apparently, the error of velocity does not even change due to $\nu$,  which is consistent with what we have been discussing in remark \ref{remark robust}. Moreover, when the viscosity is relatively small, the pressure error shows hardly change with respect to the viscosity, which aligns with Theorem \ref{thm prior u1 and p}. These numerical results demonstrate that the pressure-robust VEM scheme is accurate and robust.

\begin{table}[htbp]
        \caption{Example 1: Comparison of the standard and pressure-robust virtual element scheme for  $\nu = 1,10^{-4}, 10^{-8}, 10^{-12}$ with a fixed $\kappa=1$.}
        \centering	\label{table1 error for E1}   
        \begin{tabular}{c|ccccc|cccc}      
      \toprule  \multirow{2}{*}{$\nu$}& &\multicolumn{4}{c}{Standard VEM } &   \multicolumn{4}{c}{Pressure-robust VEM}  \\
          & $h$& $\interleave\mathbf{u}-\mathbf{u}_h\interleave$ & Rate & $\|p - p_h\|$ & Rate & $\interleave\mathbf{u}-\mathbf{u}_h\interleave$ & Rate & $\|p - p_h\|$ & Rate \\
         \hline
         \multirow{4}{*}{1} &    1.00e-01  &   1.681e+00   &   -  &   1.997e-01 &    - 	&	 1.681e+00   &   -   &  2.328e-01 &        -\\
          &  5.00e-02  &   8.175e-01  &    1.04   &  7.290e-02   &  1.45 	&	  8.172e-01  &    1.04    & 8.049e-02  &   1.53
            \\
          &  2.50e-02  &   4.042e-01   &   1.02   &  3.095e-02   &  1.24 		&  4.041e-01  &    1.02   &  3.216e-02  &   1.32
            \\
         & 1.25e-02  &   2.015e-01   &   1.00   &  1.499e-02   &  1.05 	&	  2.015e-01 &     1.00   &  1.497e-02  &   1.10
          \\
         \hline
         \multirow{4}{*}{1e-4} &    1.00e-01 &    1.688e+00  &    -&     2.157e-02    & - 	&	1.681e+00  &   -  &  2.157e-02 & - \\
         &  5.00e-02   &  8.177e-01  &    1.05  &   1.106e-02  &   0.96 &		8.172e-01   &   1.04   &  1.106e-02  &   0.96 \\
         &  2.50e-02   &  4.042e-01   &   1.02  &   5.034e-03  &   1.14 &		4.041e-01  &    1.02   &  5.034e-03  &   1.14\\
         & 1.25e-02   &  2.015e-01   &   1.00    & 2.655e-03   &  0.92 	&	2.015e-01  &    1.00   &  2.655e-03     &0.92
         \\
         \hline
         \multirow{4}{*}{1e-8} &    1.00e-01 &     $\times$  &    - &    2.157e-02   &  - 	&	1.681e+00   &   -    & 2.157e-02  &   - 
          \\
         &  5.00e-02  & $\times$   &   -    & 1.106e-02   &  0.96 &		8.172e-01  &    1.04   &  1.106e-02  &   0.96
          \\
         &  2.50e-02  & $\times$   &  -    &  5.034e-03   &  1.14 	&	4.041e-01     & 1.02   &  5.034e-03   &  1.14
         \\
         & 1.25e-02  &  $\times$   &   -   &   2.655e-03   &  0.92 	&	2.015e-01  &    1.00   &  2.655e-03   &  0.92\\
         \hline     
         \multirow{4}{*}{1e-12} &    1.00e-01 &     $\times$    &  -    &  2.157e-02 &    - &		1.681e+00    &  -  &   2.157e-02   &  - 
         \\
         &  5.00e-02 & $\times$    &  -   &   1.106e-02   &  0.96 &		8.172e-01    &  1.04   &  1.106e-02 &    0.96
         \\
         &  2.50e-02  & $\times$  &    -   &   5.034e-03 &    1.14 	&	4.041e-01    &  1.02  &   5.034e-03 &    1.14
         \\
         & 1.25e-02  &  $\times$   &  -    &  2.655e-03  &   0.92 	&	 2.015e-01   &   1.00   &  2.655e-03   &  0.92\\
         \bottomrule      
        \end{tabular}
\end{table}
}
 
\end{exam}

In the following two high contrast permeability numerical examples, a $128\times 128$ square mesh is used to test the robustness of our scheme with respect to the permeability $\kappa$.  The following two test problems have the same data setting:
$$
\Omega = (0,1)\times(0,1),\qquad \nu = 0.01,\qquad \mathbf{f} = \mathbf{0}\quad \text{and}\quad \mathbf{u} = \begin{pmatrix}
    1 \\
    0
\end{pmatrix}\; \text{on}\;\partial\Omega. 
$$
\begin{exam}
   \rm{(Fibrous structure). This example is frequently used in filtration and insulation materials. The inverse of permeability of fibrous structure is shown in Fig. \ref{Ex2 kappa}, where $\kappa^{-1}_{\min} =1$ and $\kappa^{-1}_{\max} =10^{6}$ in the red and blue regions, respectively. The first and the second components of the velocity calculated by our scheme are shown in Fig. \ref{Ex2 u1} and \ref{Ex2 u2}, respectively. The pressure profile of the pressure is presented in Fig. \ref{Ex2 pressure}. It is evident that the outcomes are congruent with those reported in the study by Zhai et al. \cite{zhai2016new}.   
 }
 \begin{figure}[htbp] 
     \centering  
     \vspace{-0.1cm} 
     \subfigtopskip=2pt 
     \subfigbottomskip=8pt 
     \subfigcapskip=0pt 
     \subfigure[$\kappa^{-1}$]{
         \label{Ex2 kappa}
         \includegraphics[width=0.30\linewidth]{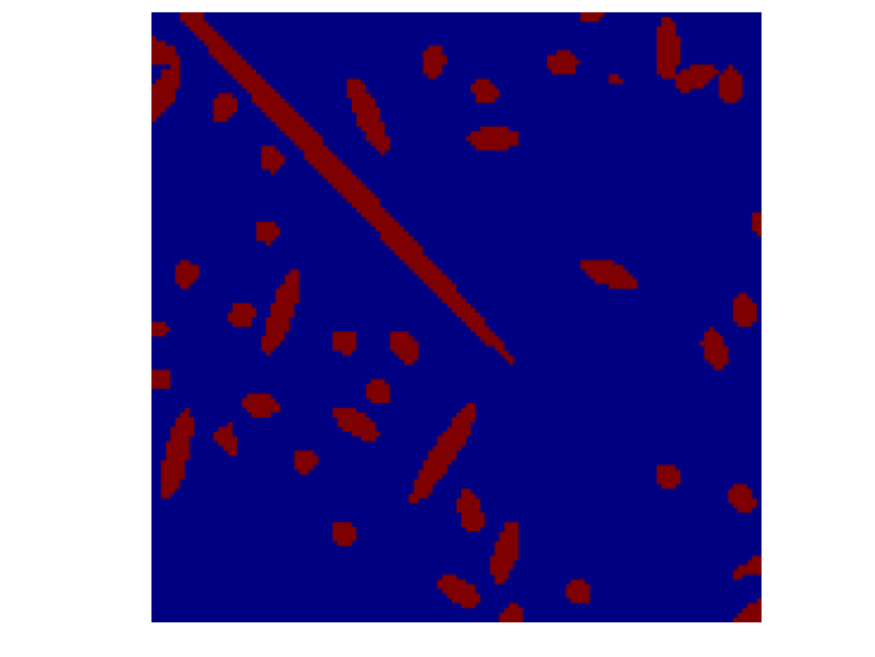}}
     \subfigure[Velocity field]{
         \label{Velocity field 6.2}
         \includegraphics[width=0.30\linewidth]{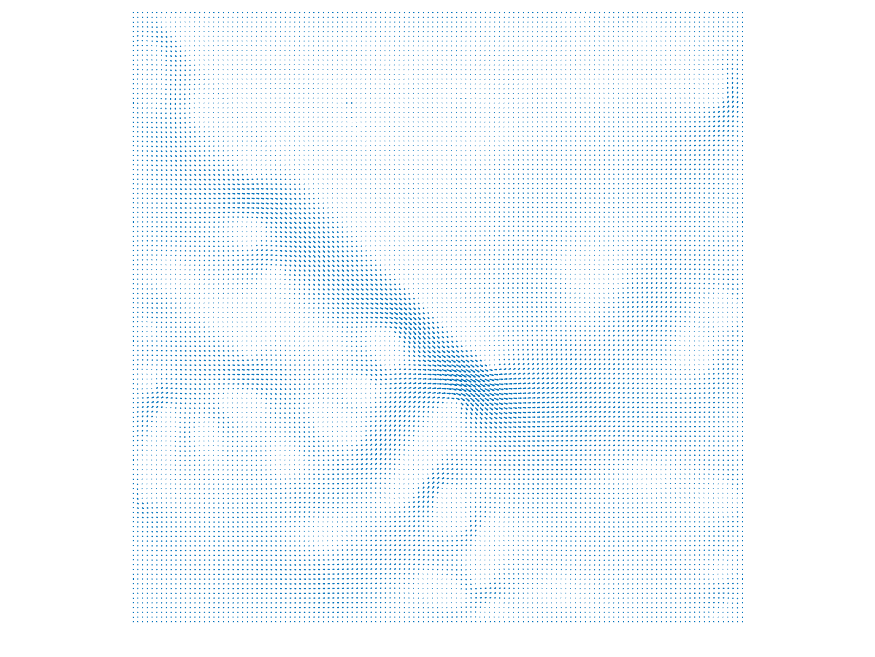}}                
  
    \subfigure[$u_{h,1}$]{
             \label{Ex2 u1}
             \includegraphics[width=0.30\linewidth]{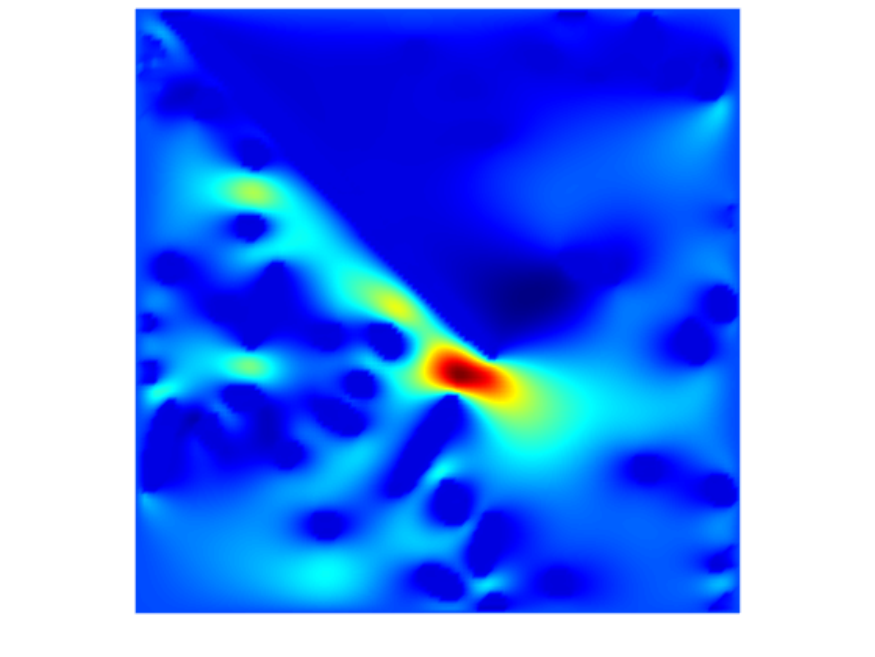}}         
     \subfigure[$u_{h,2}$]{
             \label{Ex2 u2}
             \includegraphics[width=0.30\linewidth]{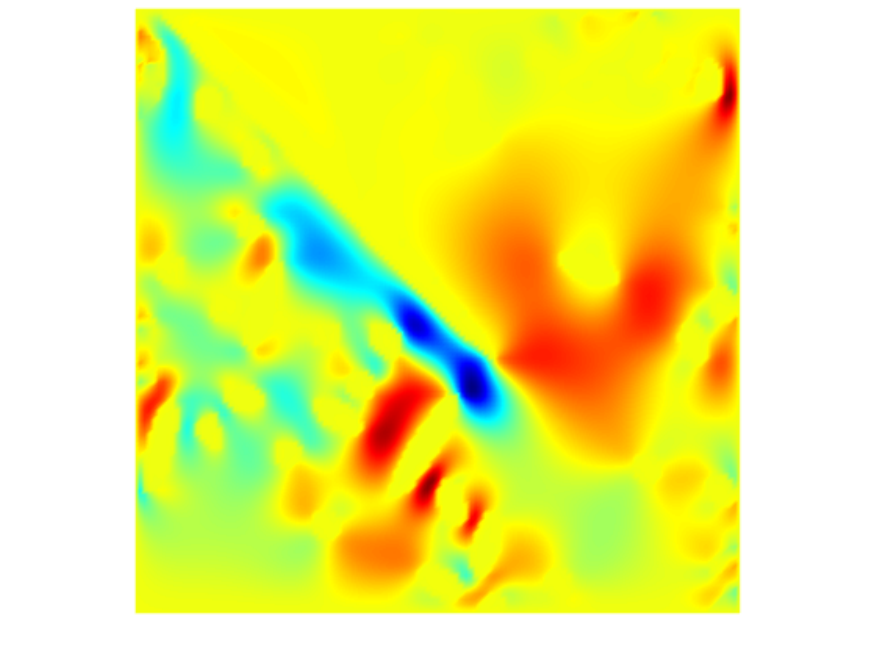}}
     \subfigure[$p_h$]{
         \label{Ex2 pressure}
         \includegraphics[width=0.30\linewidth]{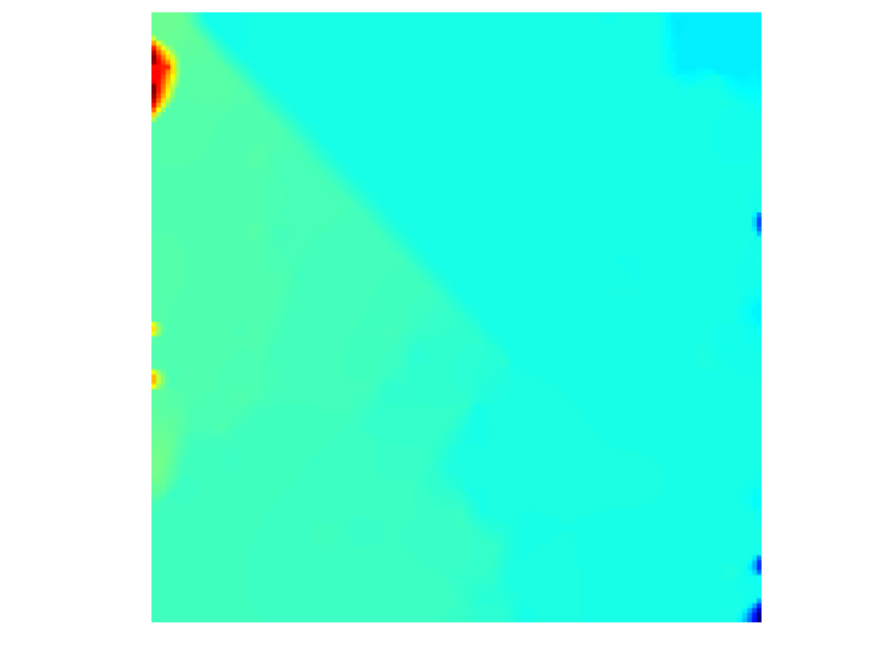}}                
     \caption{Example 6.2: Numerical results. (a) Profile of $\kappa^{-1}$ for fibrous structure; (b) Velocity field of $\mathbf{u}_h$;  (c) First component of velocity $u_{h,1}$; (d) second component of velocity $u_{h,2}$; (e) pressure profile $p_h$;}
     \label{Ex2-1}
 \end{figure}

\end{exam}

\begin{exam}
    \rm{(Open foam). The geometry of this example is an open foam with a profile of $\kappa^{-1}$ shown in Fig. \ref{Ex3 kappa}, where $\kappa^{-1}_{\min} =1$ and $\kappa^{-1}_{\max} =10^{6}$ in the red and blue regions, respectively. The profiles of the approximate pressure and velocity are presented in Fig. \ref{Ex3-1}. It is observed that the findings are analogous to those presented in the work by Wang et al. \cite{wang2021Brinkman}.
        
    }  
    
    \begin{figure}[htbp] 
        \centering  
        \vspace{-0.1cm} 
        \subfigtopskip=2pt 
        \subfigbottomskip=8pt 
        \subfigcapskip=0pt 
        \subfigure[$\kappa^{-1}$]{
            \label{Ex3 kappa}
            \includegraphics[width=0.30\linewidth]{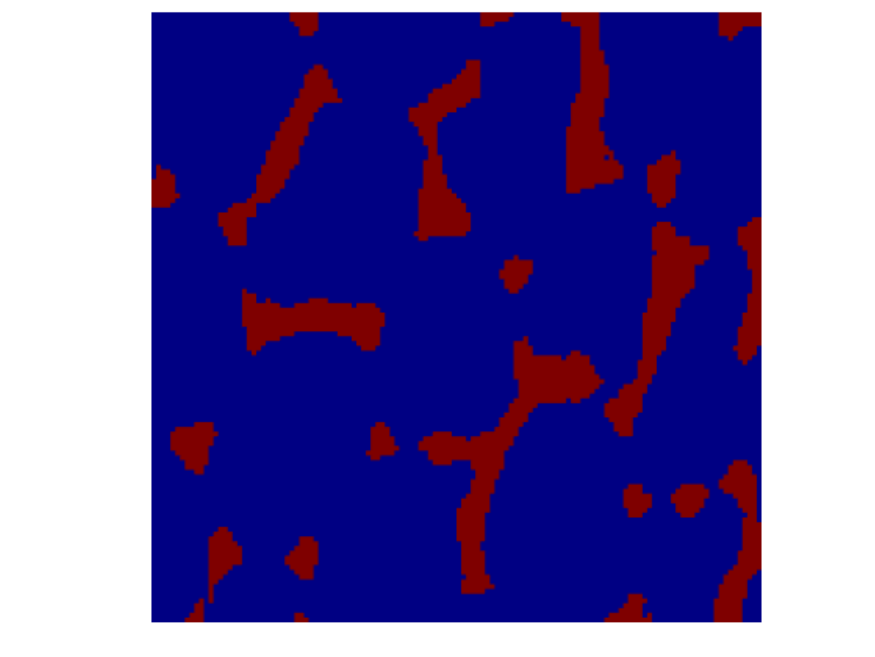}}  
        \subfigure[Velocity field]{
            \label{Velocity field 6.3}
            \includegraphics[width=0.30\linewidth]{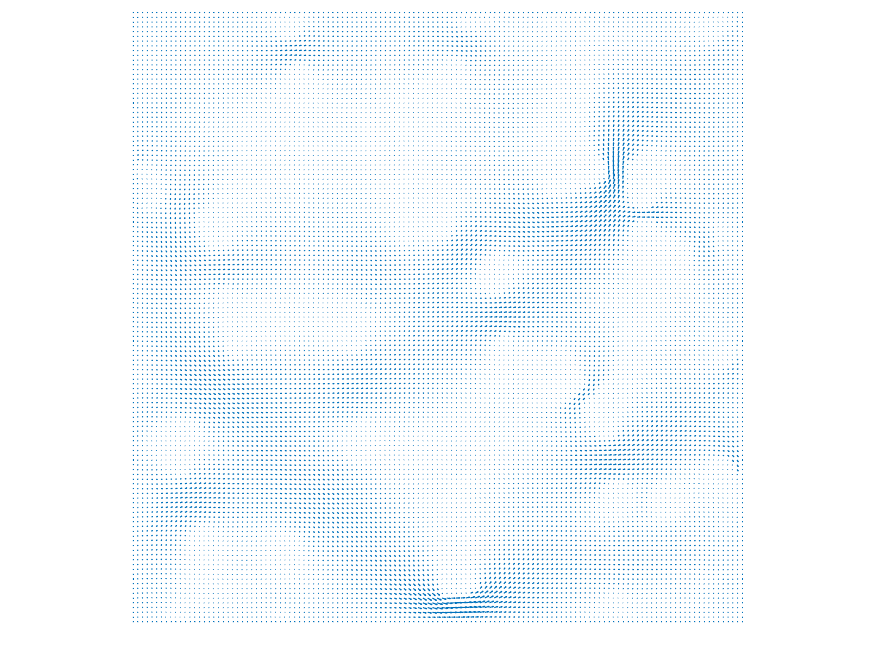}}         

          \subfigure[$u_{h,1}$]{
                \label{Ex3 u1}
                \includegraphics[width=0.30\linewidth]{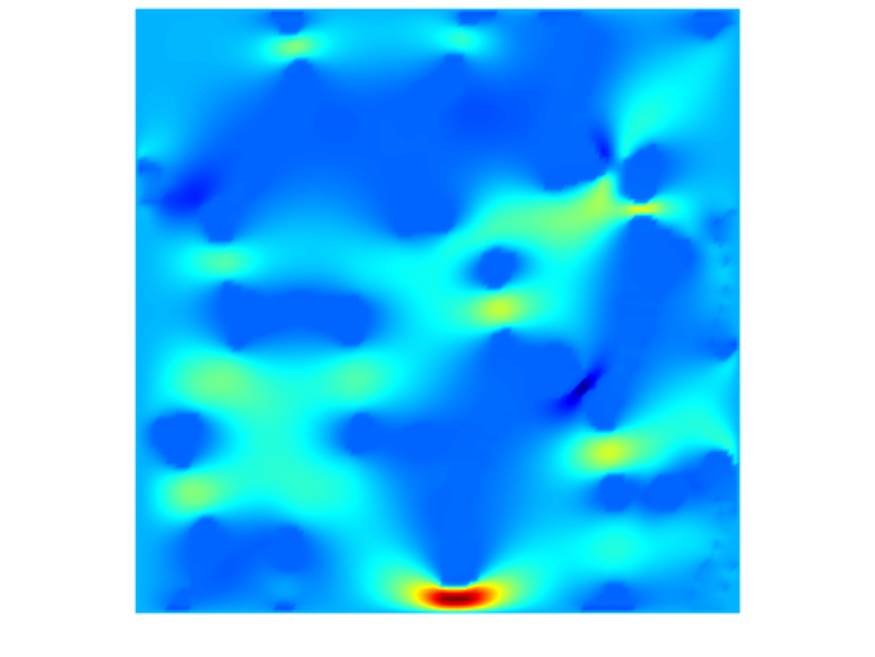}}         
        \subfigure[$u_{h,2}$]{
                \label{Ex3 u2}
                \includegraphics[width=0.30\linewidth]{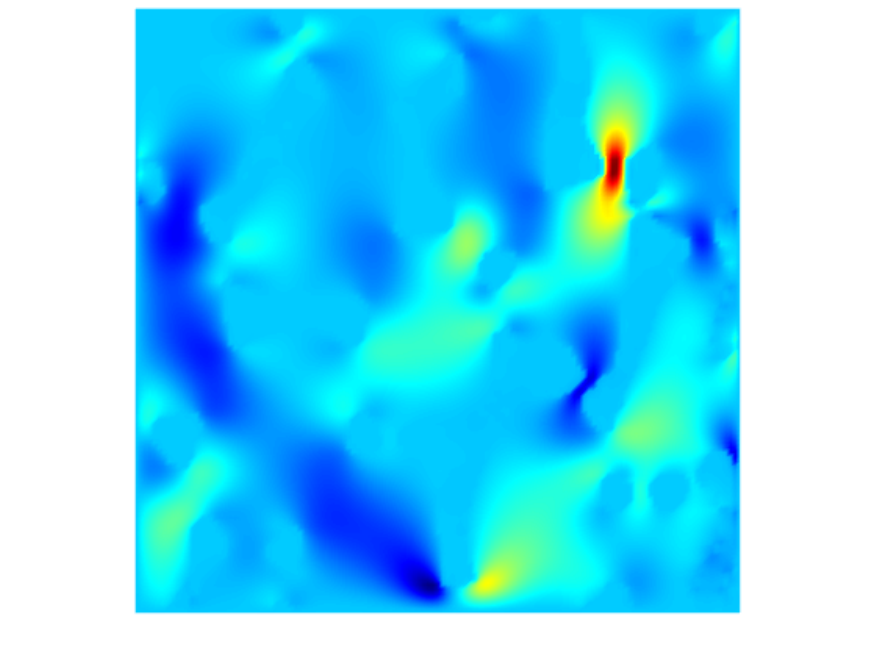}}
        \subfigure[$p_h$]{
            \label{Ex3 pressure}
            \includegraphics[width=0.30\linewidth]{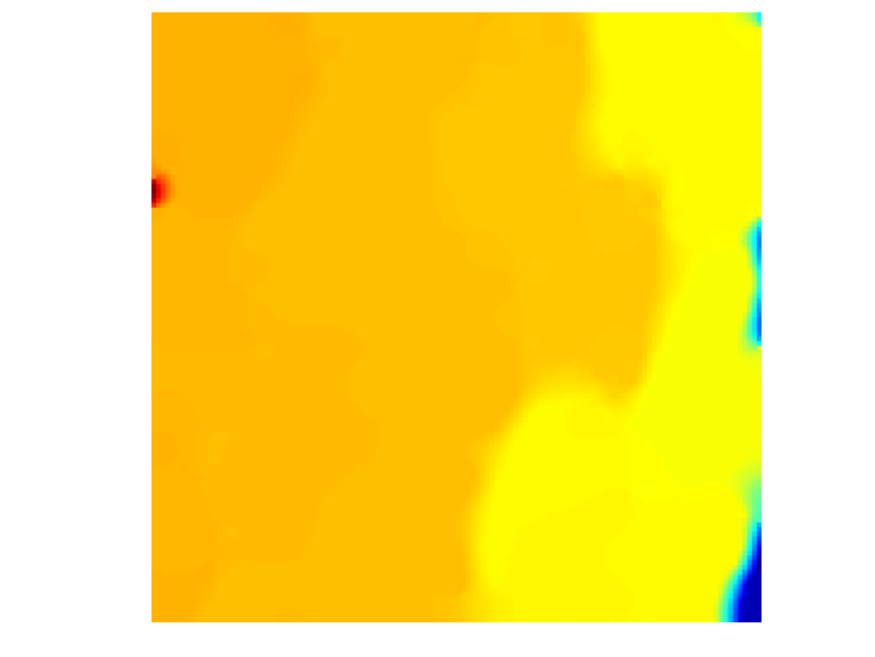}}            
        \caption{Example 6.3: Numerical results. (a) Profile of $\kappa^{-1}$ for fibrous structure; (b) Velocity field of $\mathbf{u}_h$;  (c) First component of velocity $u_{h,1}$; (d) Second component of velocity $u_{h,2}$; (e) Pressure profile $p_h$;}
        \label{Ex3-1}
    \end{figure}
  
\end{exam}
 The  velocity fields of aforementioned tow examples are shown in Fig. \ref{Velocity field 6.2} and Fig, \ref{Velocity field 6.3}. It is evident that the fluids both avoid regions of low permeability (the red part for $\kappa^{-1}$), and preferentially flow through the regions of high permeability (the blue part for $\kappa^{-1}$).

\subsection{Numerical tests for a posteriori error estimate}
In this subsection, we present three sets of examples to verify the reliability and Efficiency of the posteriori estimator $\eta$, and two examples to demonstrate the adaptive algorithm's ability to handle geometric singularities. Since the mesh is no longer a quasi-uniform polygonal mesh in our adaptive algorithm, the following numerical performance is reported with respect to the DOFs. Thus, the  convergent rates of the error $\left(\nu^2\interleave \mathbf{u} - \mathbf{u}_h \interleave^2 + \|p - p_h\|^2 \right)^{1/2}$ and $\eta$ are expected  to be $\mathcal{O}(\text{DOFs}^{-1/2})$, where $ h = \mathcal{O}(\text{DOFs}^{-1/2})$ for quasi-uniform mesh in two dimensions.

To measure the quality of the posteriori estimator, we define the effectivity index as
\begin{equation}	
    \text{Eff}: = \left(\frac{\eta^2}{\text{Err}^2_{\mathbf{u}} + \text{Err}^2_{p}}\right)^{1/2},
\end{equation}
where
$$
\text{Err}_{\mathbf{u}} := \nu\interleave \mathbf{u} - \mathbf{u}_h \interleave\quad\text{and}\quad \text{Err}_{p} = \|p - p_h\|.
$$

\begin{exam}
 \rm{The domain is $\Omega = (0,1) \times (0,1)$ and the exact solution is set to be 
    \begin{equation}
        \mathbf{u} = \left(\begin{aligned}
            10 x^2 (1-x)^2 y (1-y) (1- 2y)\\
            -10 x(1-x) (1-2x) y^2 (1-y)^2
        \end{aligned}\right),\quad p = -10 (2x -1) (2y -1),
    \end{equation}
where the load function $\mathbf{f}$ is suitably chosen, and it is easy to check that $\nabla\cdot\mathbf{u} = 0$ and $(p,1) =0$. The marking parameter $\delta$  and tolerance $\textbf{tol}$ are set equal to $0.4$ and  $10^4$, respectively. We carry out this test starting from the non-convex base mesh shown in Fig \ref{level.sub.non-convex}, and set $\nu = 1e-2,\; \kappa^{-1} = 1$.  After 13 iterations, the final results as shown in Table \ref{table error for adaptive Ex}. Moreover, we can observe  that the desired $\mathcal{O}(\text{DOFs}^{-1/2})$ optimal convergence rate is reached for both the error $\left(\text{Err}^2_{\mathbf{u}} + \text{Err}^2_{p}\right)^{1/2}$ and estimator $\eta$. The efficiency index $\text{Eff}$ is close to a constant from 7 to 8. These results validate our analytical predictions.
} 
\begin{table}[htbp]
    \caption{Example 6.4. the errors for a series of Voronoi meshes}
    \centering	\label{table error for adaptive Ex}   
    \begin{tabular}{cccccc}      
        \toprule
         DOFs & $\eta$ &Rate & $\left(\text{Err}^2_{\mathbf{u}} + \text{Err}^2_{p}\right)^{1/2}$  & Rate  & Eff \\
        \midrule
          453  &  4.2849e+00   & -    &          6.0318e-01    &          - &    7.1038\\ 
           1179  &  2.4851e+00  &  -0.57     &          3.3653e-01    &          -0.61   &  7.3846\\
          6684   & 1.0189e+00   & -0.51      &         1.2905e-01     &         -0.55   &  7.8949\\
         13794 &   7.0928e-01  & -0.50      &         8.8856e-02     &         -0.52  &   7.9823\\
          28134 &   4.9559e-01  &  -0.50   &            6.2125e-02    &          -0.50  &   7.9772\\           
        \bottomrule       
    \end{tabular}
\end{table}

\end{exam}

\begin{exam}
    \rm{(A smooth solution with boundary layers). On the unit square $\Omega= (0,1)\times(0,1)$, the exact solution is chosen as 
     \begin{equation}
             \mathbf{u} = \left(\begin{aligned}
                y - \frac{1 - e^{y/\nu}}{1- e^{1/\nu}}\\
                 x - \frac{1 - e^{x/\nu}}{1- e^{1/\nu}}
             \end{aligned}\right),\quad p =  y -x.
     \end{equation}
    The viscosity and permeability are taken as $\nu = 0.01$ and $\kappa =1$, respectively, which makes the velocity has arrow boundary layers along the line $x =1$ and $y =1$. Therefore, we expect our algorithm can refine the boundary layers. 
    
    The initial mesh is shown in Fig. \ref{initial mesh bl}. After 16 iterations, the refined mesh is obtained and shown in Fig. \ref{refined mesh bl}. We observe that there are more elements created around the boundary layers. Moreover, we observe that the desired  optimal convergence rate $\mathcal{O}(\text{DOFs}^{-1/2})$ is reached for both the error and estimator from Fig. \ref{convergence bl}. Fig. \ref{Efficiency bl} indicates the efficiency asymptotically approaches a constant value between 4 and 5.
      \begin{figure}[htbp] 
        \centering  
        \vspace{-0.1cm} 
        \subfigtopskip=2pt 
        \subfigbottomskip=5pt 
        \subfigcapskip=0pt 
        \;  \subfigure[Initial mesh]{
            \label{initial mesh bl}
            \includegraphics[width=0.30\linewidth]{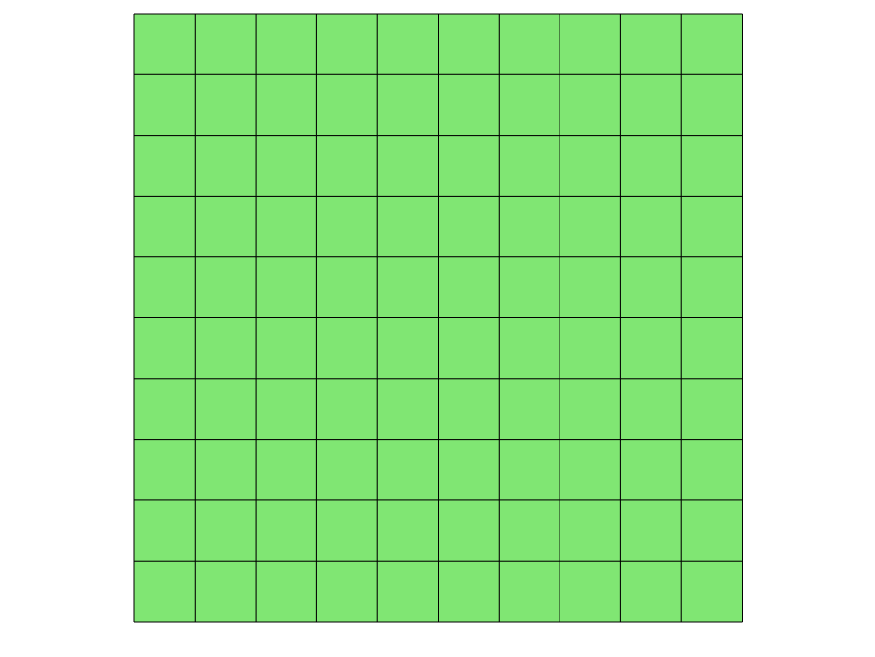}}                           
        \subfigure[Refined mesh]{
            \label{refined mesh bl}
            \includegraphics[width=0.30\linewidth]{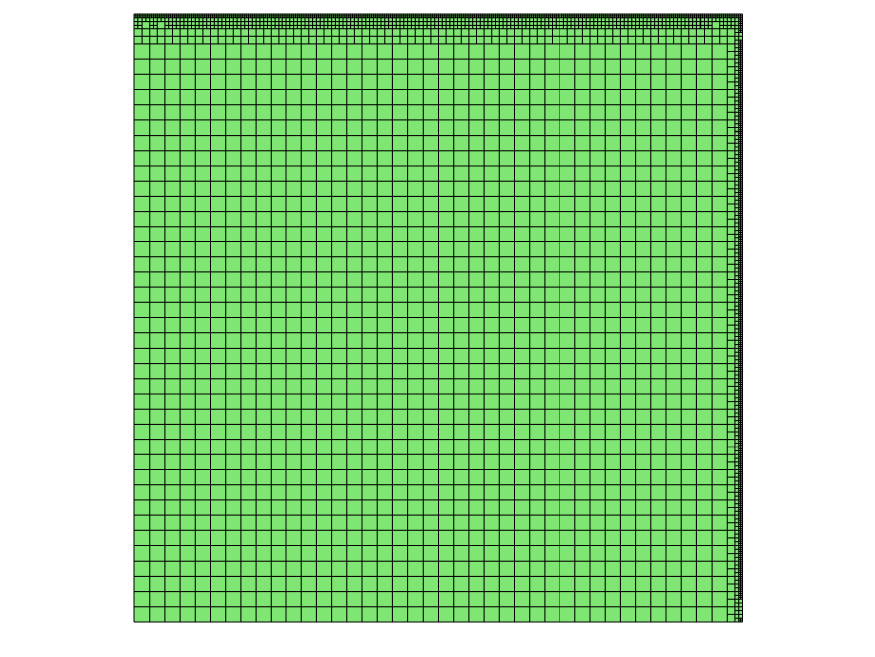}}
        
        \subfigure[Convergence]{
            \label{convergence bl}
            \includegraphics[width=0.35\linewidth]{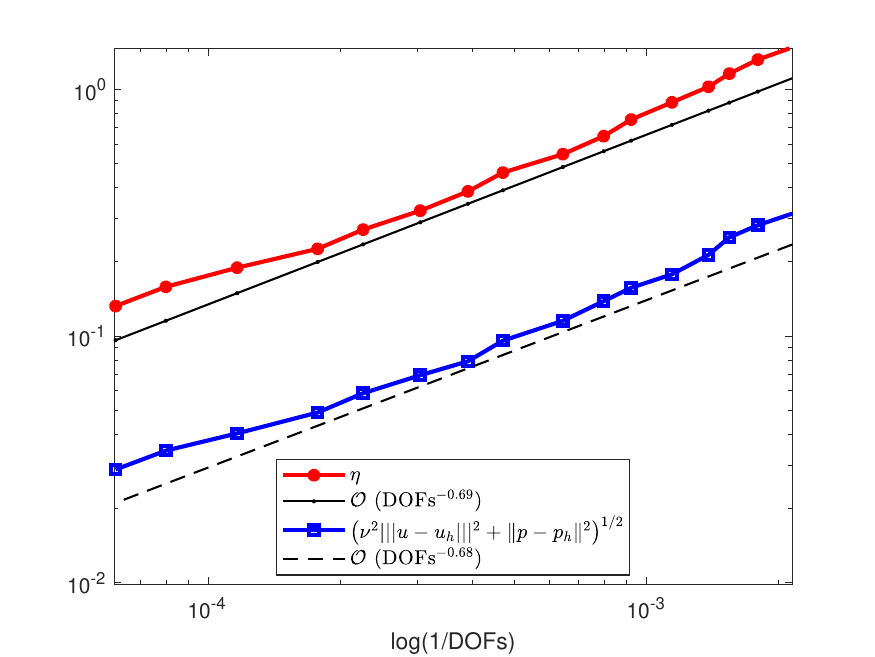}}                           
        \subfigure[Efficiency]{
            \label{Efficiency bl}
            \includegraphics[width=0.35\linewidth]{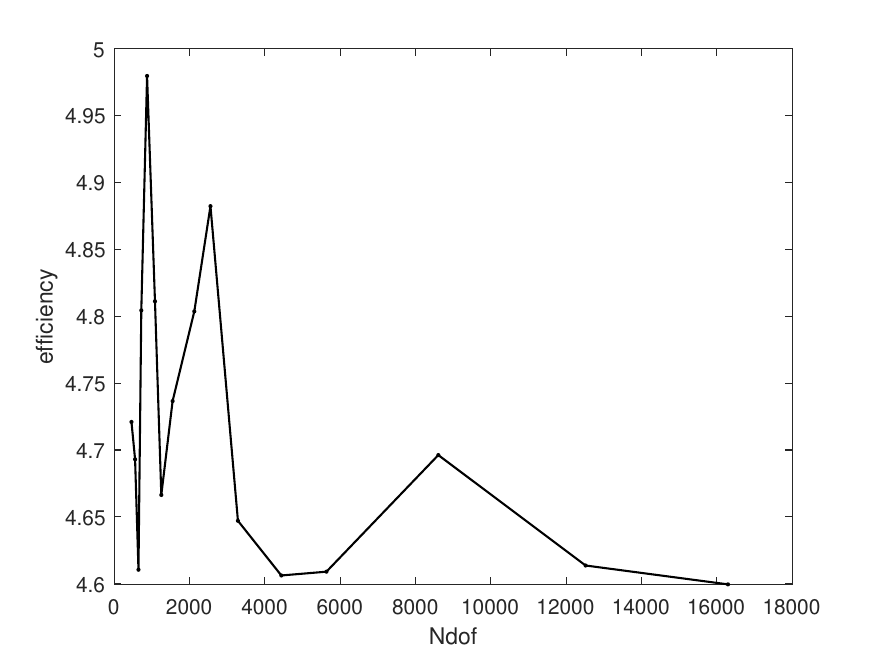}}
        \caption{Example 6.5: Numerical results. (a)Initial mesh; (b) Finally refined mesh; (c) Curves of errors convergence; (d) Efficiency index.}
        \label{Ex-bl}
    \end{figure} 

    }
\end{exam}

\begin{exam}
    \rm{(A smooth solution with interior layer). Consider the the exact solution as
  \begin{equation}
            \mathbf{u} = -10^3 e^{-10^3(1.5-x-y)^2}\begin{pmatrix}
                1\\
                -1
            \end{pmatrix}\quad \text{and}\quad p = 2 e^x \sin(y)\qquad \text{in}\quad \Omega = (0,1)\times(0,1),
 \end{equation}
 which is similar to that in \cite{barrios2015posteriori}. Setting $\nu = 0.5$, $\kappa^{-1} = 	0.01$ and the numbers of loop $\mathcal{L} = 22$.  We emphasize that there exits a inner layer around the line $1.5-x-y = 0$. We perform the test, and the numerical results as shown in Fig. \ref{Ex-5}. We can observe that the mesh refinement is concentrated around the line $1.5-x-y = 0$, which is consistent with the result we expected. We also notice that, looking at the experimental rates of convergence, the order $\mathcal{O}(\text{DOFs}^{-1/2})$ is observed for all the unknowns. Through calculating, the effective index tends to a constant.
    }
  \begin{figure}[htbp] 
      \centering  
      \vspace{-0.1cm} 
      \subfigtopskip=2pt 
      \subfigbottomskip=5pt 
      \subfigcapskip=0pt 
    \;  \subfigure[Initial mesh]{
          \label{initial mesh 5}
          \includegraphics[width=0.30\linewidth]{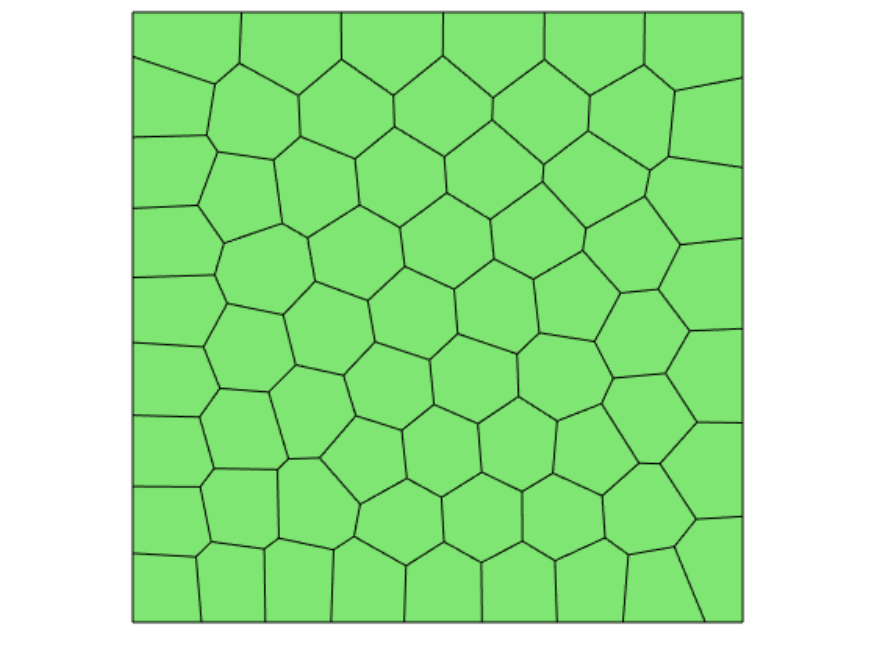}}                           
      \subfigure[Refined mesh]{
          \label{refined mesh 5}
          \includegraphics[width=0.30\linewidth]{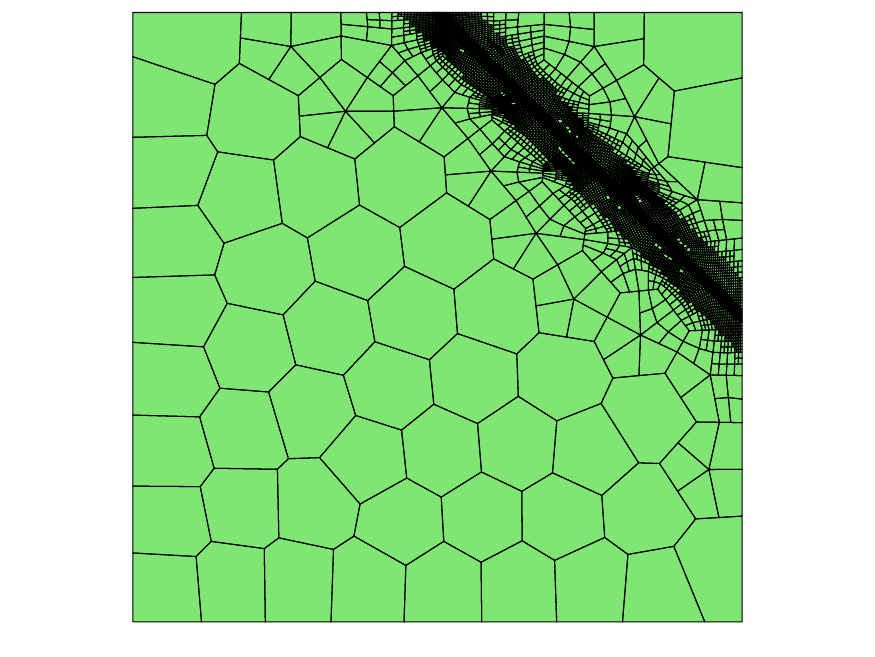}}
      
            \subfigure[Convergence]{
          \label{convergence 5}
          \includegraphics[width=0.35\linewidth]{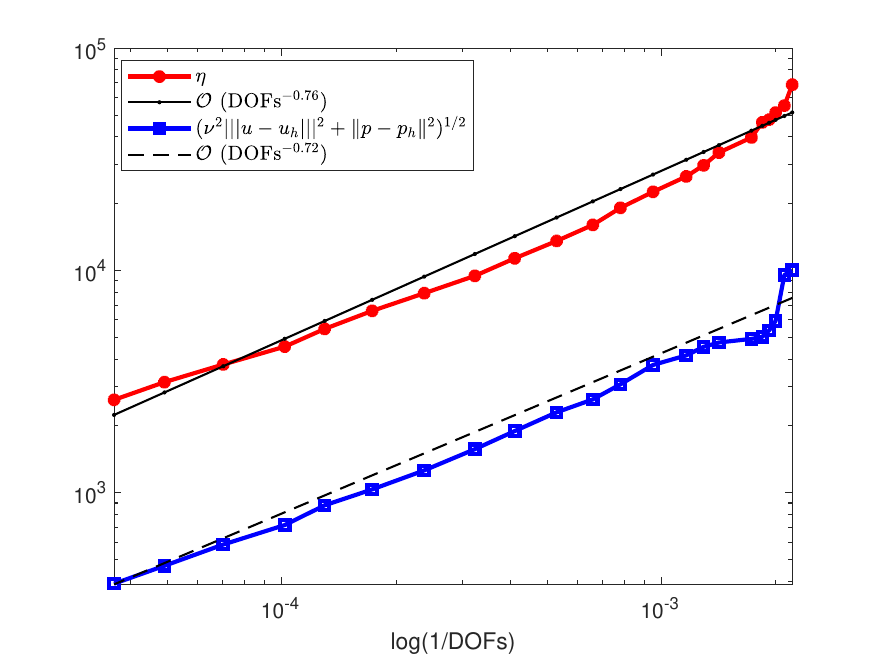}}                           
      \subfigure[Efficiency]{
          \label{Efficiency 5}
          \includegraphics[width=0.35\linewidth]{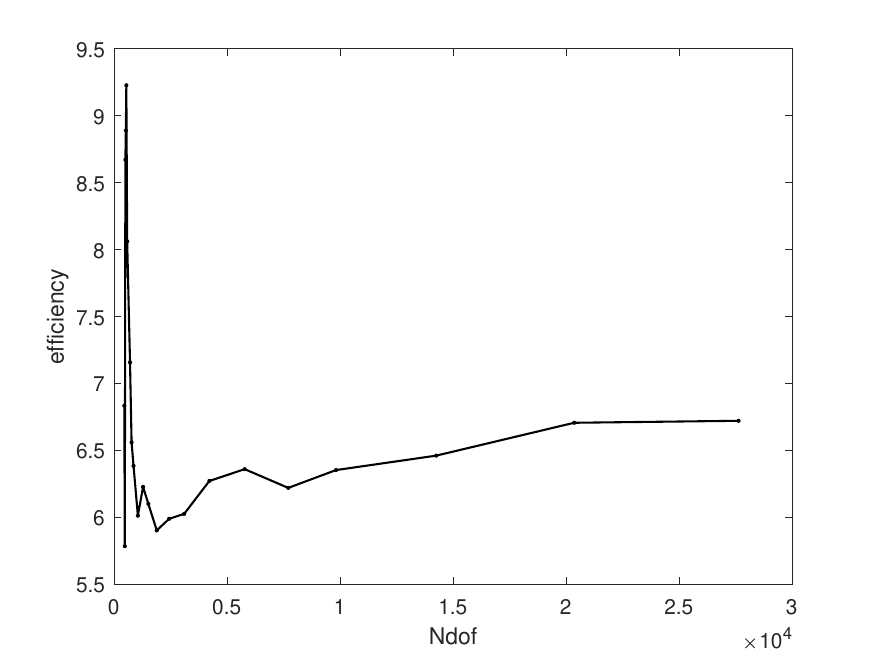}}
      \caption{Example 6.6: Numerical results. (a)Initial mesh; (b) Finally refined mesh; (c) Curves of errors convergence; (d) Efficiency index}
      \label{Ex-5}
  \end{figure}  
\end{exam}

\begin{exam}
\rm{(Corner singularity). We consider the Brinkman equation \eqref{Brinkman equations} in circular segment with radius $1$ and angle $3\pi/2$, as shown in Fig. \ref{initial mesh circle}, and with homogeneous right-hand side. The normal  and tangential  velocity on the straight parts of the  boundary are imposed and no slip condition is enforced on the remaining walls, i.e., 
     \begin{equation}
        \mathbf{u} = \begin{cases}
            (0,1)^{\top} \quad &\text{on straight left lid,} \\
            (-1,0)^{\top} \quad&\text{on straight top lid,}\\
            \mathbf{0} \quad&\text{on otherwise curved wall.}
        \end{cases}
    \end{equation}
The viscosity and permeability are taken by $1e-3$ and $1e3$. We can notice that the boundary data of velocity is discontinuous. The velocity $\mathbf{u}$ has  discontinuities at the origin  and two corners of the boundary. The so-called corner singularity appears where the straight lids meet the curved wall. Therefore, we expect the mesh refinement to be concentrated at the origin  and two corners. After 15 iterations, the refined mesh and its detail are depicted in Figs. \ref{adaptive mesh circle L}-\ref{detail 6}. We see that most of the elements are located at the origin and two corners, which confirms that our method captures the  corner singularity. Moreover, Figs. \ref{Velocity field 6}-\ref{Volocity circle L} show a good agreement between our numerical results and the benchmark data.
      
      \begin{figure}[htbp] 
          \centering  
          \vspace{-0.1cm} 
          \subfigtopskip=2pt 
          \subfigbottomskip=5pt 
          \subfigure[Initial mesh]{
              \label{initial mesh circle}
              \includegraphics[width=0.30\linewidth]{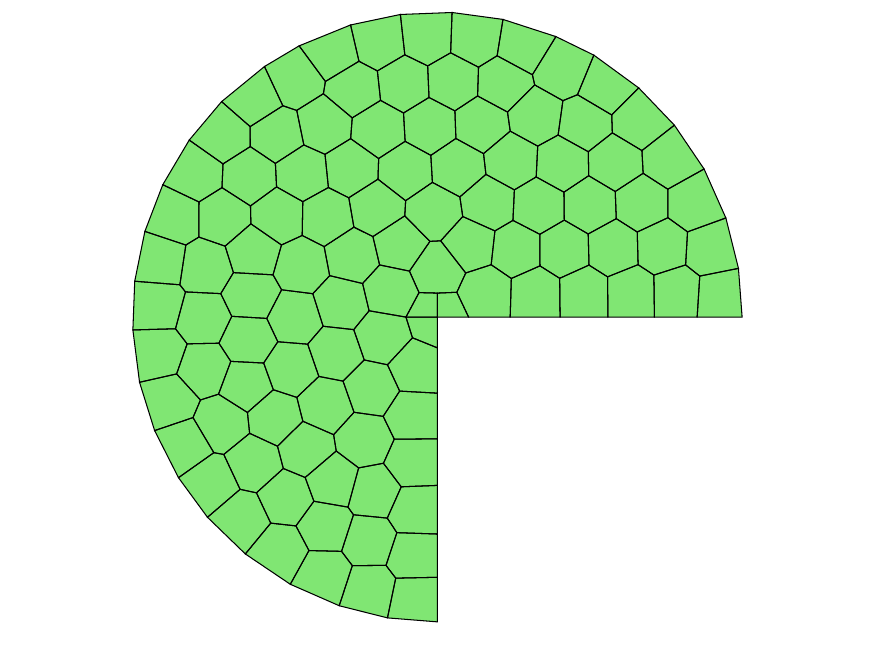}}  
          \subfigure[Adaptive mesh]{
              \label{adaptive mesh circle L}
              \includegraphics[width=0.30\linewidth]{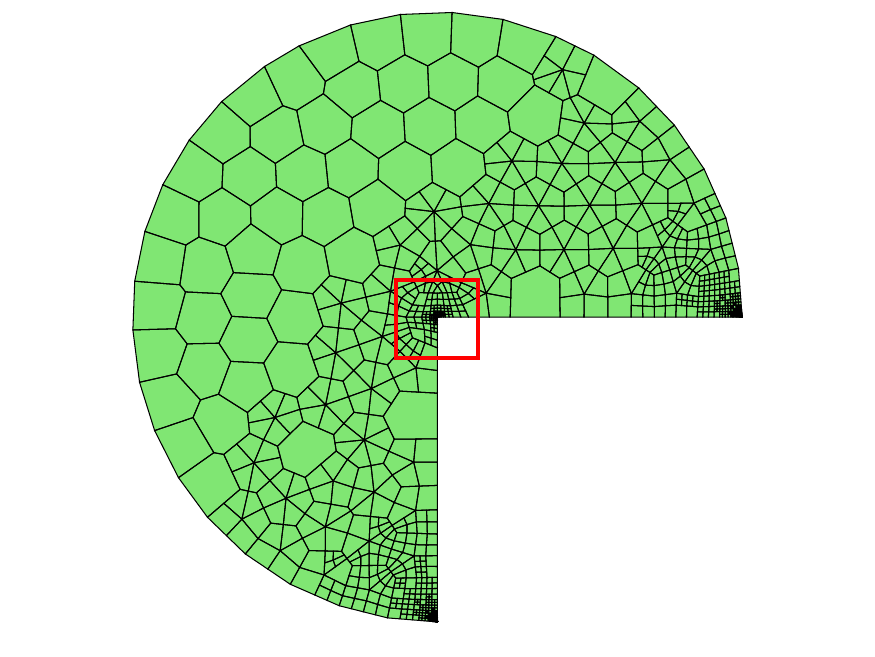}} 
          \subfigure[Local detail]{
              \label{detail 6}
              \includegraphics[width=0.30\linewidth]{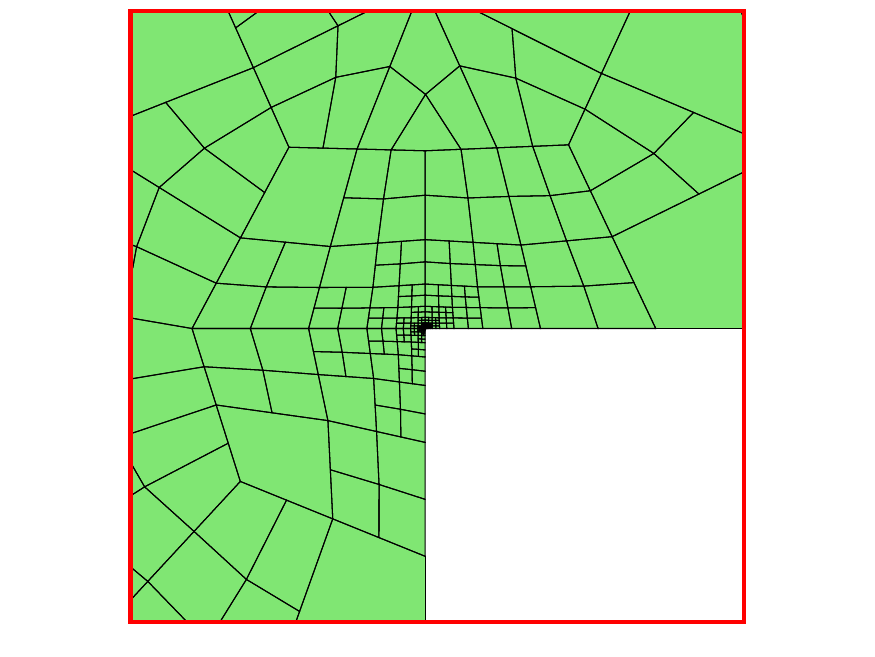}}
            
            \subfigure[Velocity field]{
                \label{Velocity field 6}
                \includegraphics[width=0.30\linewidth]{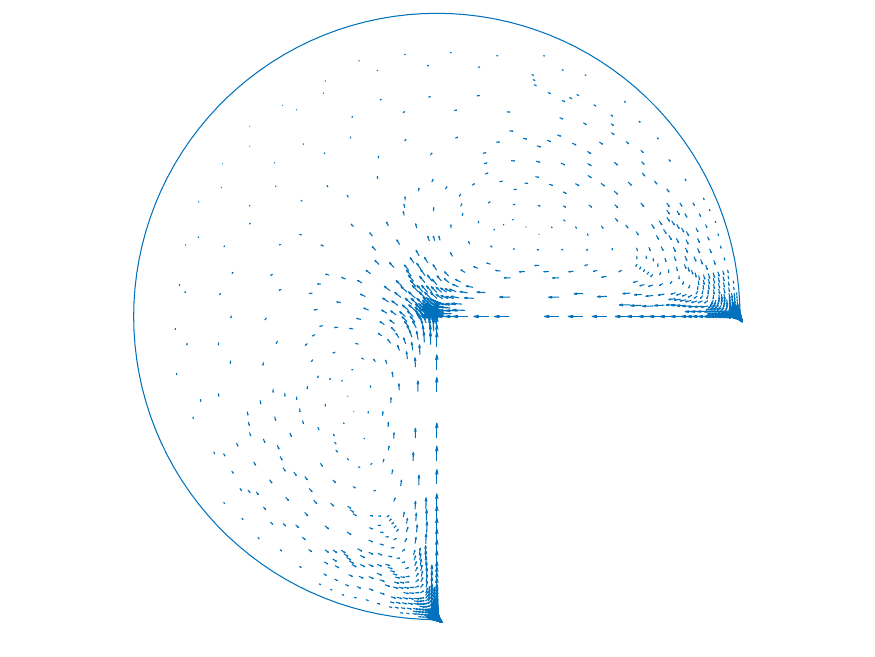} 
            }      \hspace{1cm}               
            \subfigure[Contour of velocity]{
              \label{Volocity circle L}
              \includegraphics[width=0.30\linewidth]{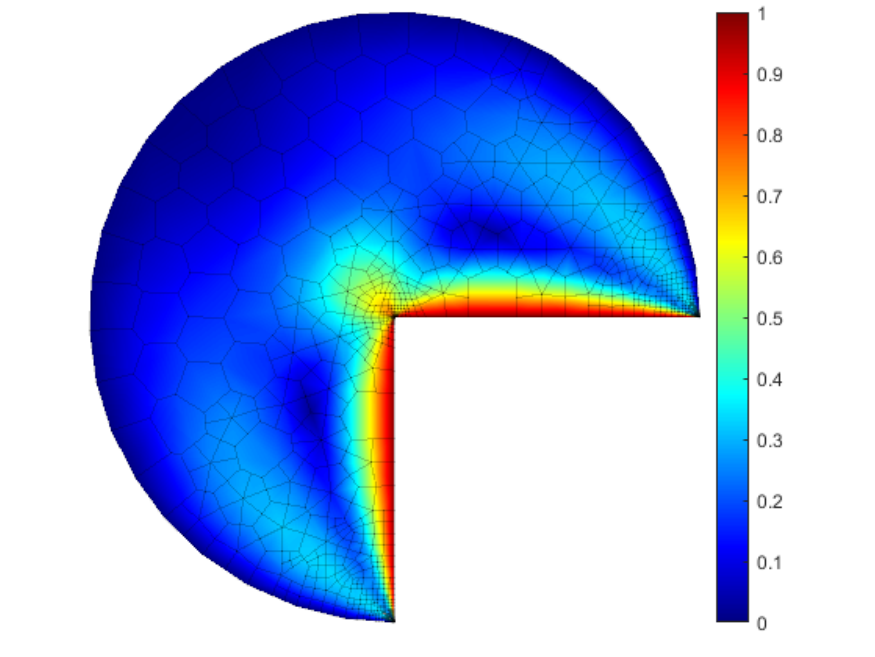} 
              }   
     \caption{Example 6.7: Numerical results. (a) Initial mesh; (b) Adaptive refined mesh; (c)  Local Refined detail around origin; (d) Velocity field of $\mathbf{u}_h$; (e) Contour plot  of $\mathbf{u}_h$.   }                 
          \label{Ex-6}
      \end{figure}    
 }
\end{exam}

\begin{exam}
        \rm{(Channel flow past a cylinder). For our last experiment, we consider the adaptive algorithm applied to the common benchmark problem of 2D channel flow past a cylinder with viscosity $\nu = 10^{-3}$ and permeability $\kappa = 1$. The domain is a $2.2 \times 0.41$ rectangular channel with a cylinder of radius $0.1$ centered at $(0.2, 0.2)$, see Fig. \ref{Computational domain 6.8}. There is no external forcing, and no-slip boundary conditions are prescribed for the walls and the cylinder, while the inflow and outflow profiles are given by
            \begin{equation}
                \mathbf{u} = \left(\begin{aligned}
                    \frac{6}{0.41^2} \sin(\pi/&8) y (0.41-y)\\
                    0&
                \end{aligned}\right).
            \end{equation}
            Around the cylinder, the flow varies considerably. Therefore, we expect the mesh refinement concentrates around the cylinder. 
            
            The initial computational mesh is depicted in Fig. \ref{Initial mesh Cylinder}, which is generated by PolyMesher \cite{polymesher2012}. After 10 iteration, the finally adaptive mesh and contour of velocity are displayed in Figs. \ref{Finally mesh velocity} - \ref{Contour of velocity}. We observe that the local mesh refinement is mainly carried out near the cylinder, which validates that our method captures the  Geometric singularity.
        } 
        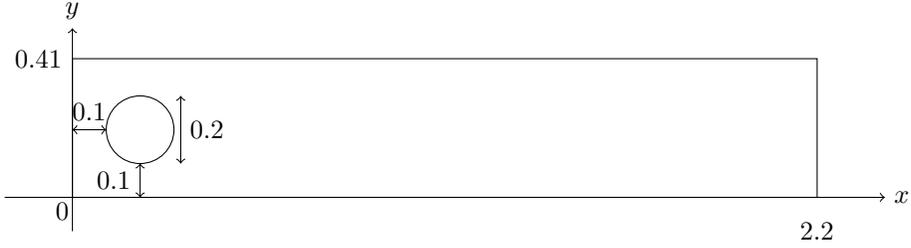
\begin{figure}[h]
            \centering
           \qquad\; \begin{tikzpicture}[scale=4.5]
                \draw[->] (-0.2,0) -- (2.4,0) node[right] {$x$};
                \draw[->] (0,-0.1) -- (0,0.5) node[above] {$y$};
                
                \draw (0,0) rectangle (2.2,0.41);
                
                \draw (0.2,0.2) circle (0.1);
                
                \draw[<->] (0.2,0) -- (0.2,0.1) node[midway,left] {0.1};

                \draw[<->] (0,0.2) -- (0.1,0.2) node[midway,above] {0.1};
                \draw[<->] (0.32,0.1) -- (0.32,0.3) node[midway,right] {0.2};
                
                \node at (2.2,-0.1) {2.2};
                \node at (-0.1,0.41) {0.41};
                \node at (-0.03,-0.04) {0};           
            \end{tikzpicture}
            \caption{Example 6.8: Computational domain.}	\label{Computational domain 6.8}
        \end{figure}
        
  \begin{figure}[htbp] 
    \centering  
    \vspace{-0.1cm} 
    \subfigtopskip=2pt 
    \subfigbottomskip=5pt 
    \subfigure[Initial mesh]{
        \label{Initial mesh Cylinder}
        \includegraphics[width=0.8\linewidth]{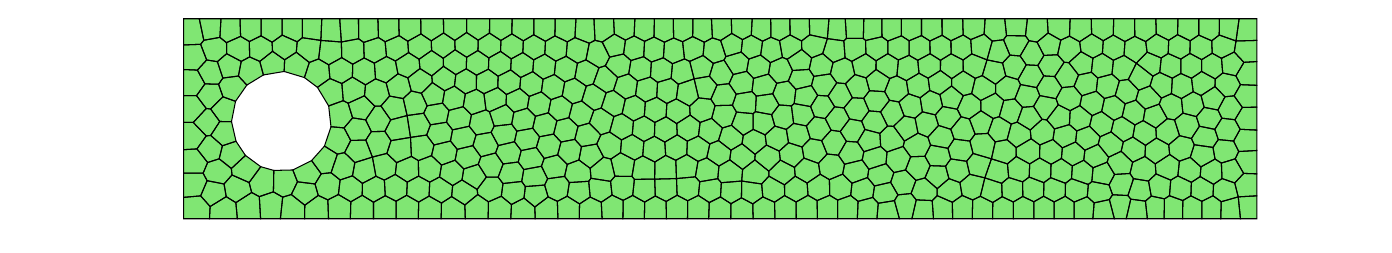}} 
    
    \subfigure[Refined mesh]{
        \label{Finally mesh velocity}
        \includegraphics[width=0.8\linewidth]{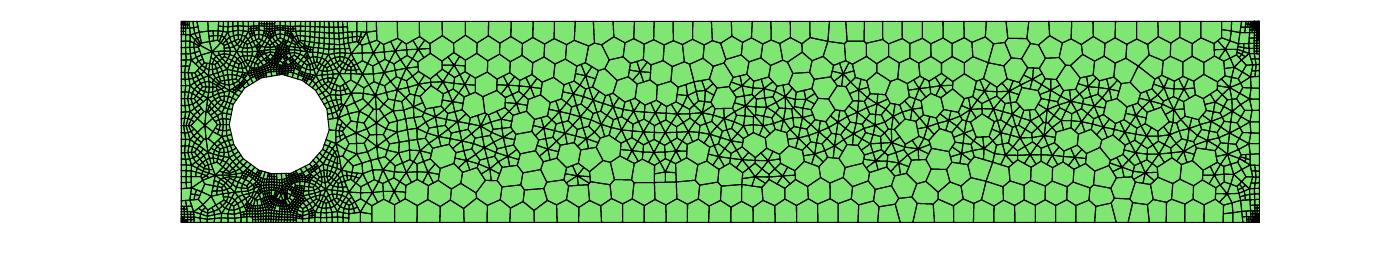}}
        
   \quad\;\;\,\,\, \subfigure[Contour of velocity]{
            \label{Contour of velocity}
            \includegraphics[width=0.85\linewidth]{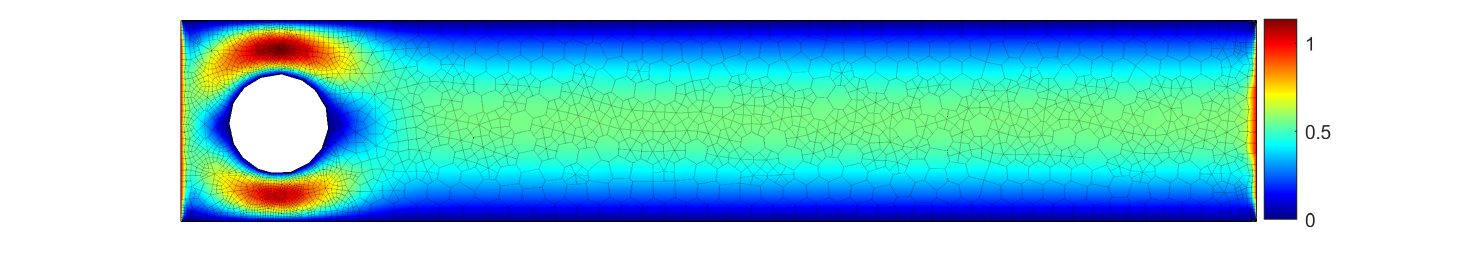}}
    \caption{Example 6.8: Numerical results. (a) Initial mesh; (b) Adaptive refined mesh; (c) Contour of velocity. }                 
    \label{Ex-Cylinder_domian}
  \end{figure}

\end{exam}

\section{conclusion}
In this paper, we have developed and analyzed a really pressure-robust VEM for solving the incompressible Brinkman problem \ref{Brinkman equations}. Based on the  space divergence-preserving reconstructor $\mathcal{R}_h$, a well-posed and really pressure-robust VEM scheme is constructed. The velocity error in the energy norm and the pressure error in the $L^2$ norm are proven to be optimal. It should be emphasized that the velocity error does not explicitly depend on the pressure and even viscosity $\nu$. Therefore, the locking phenomenon will not occur as $\nu\rightarrow 0$. To drive mesh adaptivity, we have designed a residual-based  posteriori error indicator tailored for the Brinkman equations, and have demonstrated its effectiveness and reliability. Finally, some numerical experiments on three-type polygonal meshes are given to validate our theoretical analyses.

\section*{Acknowledgements}
 Yu Xiong was supported by Postgraduate Scientific Research Innovation Foundation of Xiangtan University (XDCX2024Y178). Yanping Chen was supported by the State Key Program of National Natural Science Foundation of China (11931003), Natural Science Research Start-up Foundation of Recruiting Talents of Nanjing University of Posts and Telecommunications (NY223127). This work is grateful to associate professor Gang Wang from Northwestern Polytechnical University for his constructive suggestions.

\clearpage
\bibliographystyle{plain}
\bibliography{VEM_Brinkman}

\end{document}